\documentclass[12pt,oneside]{amsart}

\title{Deformation Theory of the Trivial mod $p$ Galois Representation for $\mathrm{GL}_n$}
\author{Ashwin Iyengar}
\address{King's College London}
\email{ashwin.iyengar@kcl.ac.uk}

\newcommand*{\backref}[1]{}
\newcommand*{\backrefalt}[4]{%
  \ifcase #1 %
    \relax
  \or
    $\uparrow$ #2.%
  \else
    $\uparrow$ #2.%
  \fi%
}

\usepackage[backref,linktoc=all]{hyperref}
\hypersetup{
    colorlinks,
    citecolor = cyan,
    filecolor = black,
    linkcolor = magenta,
    urlcolor  = black
}

\usepackage[alphabetic]{amsrefs}

\usepackage[inner = 2.5cm, outer = 2.5cm, bottom = 3cm, footskip = 1cm]{geometry}
\usepackage{amsmath, bm, amssymb, amsthm, graphicx, mathrsfs, eucal, cancel,stmaryrd}
\usepackage{tikz-cd}

\usepackage[shortlabels]{enumitem}

\usepackage{textgreek}

\usepackage[utf8]{inputenc}
\usepackage[T1]{fontenc}

\newtheorem{theorem}{Theorem}[section]
\newtheorem{corollary}[theorem]{Corollary}
\newtheorem{lemma}[theorem]{Lemma}
\newtheorem{proposition}[theorem]{Proposition}
\newtheorem{conjecture}[theorem]{Conjecture}

\theoremstyle{definition}
\newtheorem{definition}[theorem]{Definition}
\newtheorem{remark}[theorem]{Remark}

\DeclareMathOperator{\Spec}{Spec}
\DeclareMathOperator{\Spf}{Spf}
\DeclareMathOperator{\ad}{ad}

\begin{document}

\begin{abstract}
We study the rigid generic fiber $\mathcal{X}^\square_{\overline\rho}$ of the framed deformation space of the trivial representation $\overline\rho: G_K \to \mathrm{GL}_n(k)$ where $k$ is a finite field of characteristic $p>0$ and $G_K$ is the absolute Galois group of a finite extension $K/\mathbf{Q}_p$. Under some mild conditions on $K$ we prove that $\mathcal{X}^\square_{\overline\rho}$ is normal. When $p > n$ we describe its irreducible components, and show Zariski density of its crystalline points.
\end{abstract}

\maketitle

\tableofcontents

\section{Introduction}

The $p$-adic local Langlands correspondence for $\mathrm{GL}_n(\mathbf{Q}_p)$ is a hypothetical correspondence between continuous unitary $L$-Banach space representations of $\mathrm{GL}_n(\mathbf{Q}_p)$ and continuous representations $\mathrm{Gal}(\overline{\mathbf{Q}_p}/\mathbf{Q}_p) \to \mathrm{GL}_n(L)$, where $L$ is a $p$-adic local field. Thus far, such a correspondence is only known to exist when $n = 1,2$. When $n = 1$ the construction boils down to local class field theory, and when $n = 2$ such a correspondence has been defined and studied by Colmez in \cite{Col2}. A candidate correspondence (which \textit{a priori} may depend on various choices) for $\mathrm{GL}_n(\mathbf{Q}_p)$ in one direction has been defined in \cite{Car} when $p \nmid 2n$.

Crystalline Galois representations are in the image of Colmez's correspondence for $\mathrm{GL}_2(\mathbf{Q}_p)$. To show that the correspondence is surjective, Kisin suggested that one should try to show that the crystalline representations form a Zariski dense subset of the space of all representations. This was carried out in \cite{Kis10} and \cite{Col2}, except for when $p = 2$ and the $\mathrm{Gal}(\overline{\mathbf{Q}_2}/\mathbf{Q}_2)$ representation has trivial mod $2$ reduction up to semisimplification and twisting by a character. This was resolved in \cite{CDP} by studying the universal deformation ring of the mod 2 trivial representation.

In this paper, we describe the framed deformation ring of the mod $p$ trivial representation of $\mathrm{Gal}(\overline{K}/K)$ for $K$ a $p$-adic local field where now $p$ and $n$ are allowed to vary. Assuming $p > n$ and assuming some mild conditions on $K$, we first give a description of its irreducible components, which answers a particular case of a question of B\"ockle and Juschka written in the introduction to \cite{BJ}. We then prove that crystalline points are dense in the deformation space associated to the trivial representation. We hope that one day this result could be similarly useful once a $p$-adic local Langlands correspondence for $\mathrm{GL}_n(K)$ has been constructed and developed. Since our results are about unrestricted local deformation spaces at $p$, we also hope the results could be useful when proving automorphy lifting theorems for $p$-adic automorphic forms.

In the rest of the introduction we will state the two main theorems more precisely and briefly discuss the methods used in the proofs.

\subsection{Main Results}

Fix a prime $p$, and finite extensions $K/\mathbf{Q}_p$ and $L/\mathbf{Q}_p$. Let $\mathscr{O}_L$ denote the ring of integers of $L$ with maximal ideal $\mathfrak{m}_L$ and uniformizer $\varpi_L$, and let $k_L = \mathscr{O}_L/\mathfrak{m}_L$ denote the residue field. Let $\mu_{p^\infty}(F)$ denote the $p$-power roots of unity in any field $F$. Throughout this paper we will assume $|\mu_{p^\infty}(L)| \geq |\mu_{p^\infty}(K)|$.

Given any continuous residual representation $\overline\rho: G_K := \mathrm{Gal}(\overline{K}/K) \to \mathrm{GL}_n(k_L)$, we can define the framed deformation functor $D_{\overline\rho}^\square$ on the category of complete local Noetherian $\mathscr{O}_L$-algebras with residue field $k_L$, which is always represented by a complete local Noetherian ring $R^\square_{\overline\rho}$.

Let $\mathcal{X}^\square_{\overline\rho}$ denote the rigid generic fiber associated to $\Spf R^\square_{\overline\rho}$ in the sense of Berthelot (see \cite[Section 7]{DeJ} for a construction). In \cite{BJ}, B\"ockle and Juschka pose the following question about the geometry of $\mathcal{X}^{\square}_{\overline\rho}$, which we state as a conjecture (this is actually a slight rephrasing of what they write: see Remark \ref{caveats}(1)).

\begin{conjecture}[{{\cite[Question 1.10]{BJ}}}]\label{BJ}
For any prime $p$, any positive integer $n$, any finite extension $K/\mathbf{Q}_p$ and any continuous $\overline\rho: G_K \to \mathrm{GL}_n(k_L)$, let $d: \mathcal{X}^{\square}_{\overline\rho} \to \mathcal{X}_{\det\overline\rho}^\square$ denote the map induced by mapping a deformation of $\overline\rho$ to its determinant. Then $d$ induces a bijection between the irreducible components of both spaces, and the irreducible and connected components of $\mathcal{X}^\square_{\overline\rho}$ coincide.
\end{conjecture}

This should work as follows. Fix a generator $\zeta_K \in \mu_{p^\infty}(K)$ and an element $\zeta \in \mu_{p^\infty}(L)$ and define the subfunctor
	\[ D^{\square,\zeta}_{\overline\rho}(R) = \{\rho \in D^{\square}_{\overline\rho}(R) : \det\rho(\mathrm{rec}_K(\zeta_K)) = \zeta\} \subset D_{\overline\rho}^\square(R) \]
where $\mathrm{rec}_K$ is the local Artin reciprocity map. Then $D^{\square,\zeta}_{\overline\rho}(R)$ is relatively representable, and one should be able to show that the generic fiber $\mathcal{X}^{\square,\zeta}_{\overline\rho}$ is irreducible. Thus the components should be parametrized by the set $\mu_{p^\infty}(K)$.

Our first main result is a special case of the conjecture, and the proof is given in Section \ref{theproof}:

\begin{theorem}\label{FirstMainTheorem}
Assume $K$ contains a primitive $4$th root of unity if $p = 2$. Let $\overline\rho: G_K \to \mathrm{GL}_n(k_L)$ be the trivial representation. Then $\mathcal{X}_{\overline\rho}^\square$ is normal, and in particular its connected components are irreducible.

If in addition either $p > n$ or $\mu_{p^\infty}(K) = \{1\}$, then the determinant map $\mathcal{X}^{\square}_{\overline\rho} \to \mathcal{X}_{\det\overline\rho}^\square$ induces a bijection on irreducible components, which are canonically in bijection with $\mu_{p^\infty}(K)$.
\end{theorem}

Conjecture \ref{BJ} is resolved in many other cases already, and we give a brief survey of results here.

\begin{enumerate}
	\item When $n = 1$, the theorem follows straightforwardly from local class field theory. We recall this short and explicit computation in Section \ref{deforming1}.
	\item In the same paper \cite{BJ}, B\"ockle and Juschka verify Conjecture \ref{BJ} when $p > 2$, $n = 2$, $K/\mathbf{Q}_p$ is any finite extension, and $\overline\rho$ is any continuous representation.
	\item The main result of \cite{CDP} is that Conjecture \ref{BJ} is true when $p = 2$, $n = 2$, $K = \mathbf{Q}_2$, and $\overline\rho$ is the trivial representation.
	\item Let $\epsilon: G_K \to k_L^\times$ denote the mod $p$ cyclotomic character. If $p$, $n$, and $K$ are arbitrary and $H^0(G_K,\ad^0\overline\rho \otimes \epsilon) = 0$, then the proof of \cite[Theorem 4.4]{Nak} shows that Conjecture \ref{BJ} is true: in this case the irreducible components are actually smooth. In the paper Nakamura also requires that $\overline\rho$ has only scalar endomorphisms, but this assumption can be removed by adding framings.
\end{enumerate}

A consequence of the description of the irreducible components of $\mathcal{X}^{\square}_{\overline\rho}$ is that its crystalline points, i.e. points whose induced representations are crystalline, are dense in the whole space. More specifically, one first shows that the Zariski closure of the crystalline points are the union of \textit{some} irreducible components, so it then suffices to describe all of the components and write down a crystalline point in each one. For technical reasons, it is usually more convenient to show that \textit{regular} crystalline points are dense, which means that the labeled Hodge-Tate weights are distinct (cf. Definition \ref{benign}). We are thus led to the following conjecture.

\begin{conjecture}\label{DensityConj}
For $(p,n,K,\overline\rho)$ as in Conjecture \ref{BJ}, the set of regular crystalline points is dense in $\mathcal{X}^{\square}_{\overline\rho}$.
\end{conjecture}

As a consequence of Theorem \ref{FirstMainTheorem} we come to our second main result.

\begin{theorem}[Theorem \ref{SecondMainTheorem}]\label{secondref}
Assume $K$ contains a primitive $4$th root of unity if $p = 2$, and assume either $p > n$ or $\mu_{p^\infty}(K) = \{1\}$. Let $\overline\rho$ be the trivial representation. Then the set of regular crystalline points is dense in $\mathcal{X}^\square_{\overline\rho}$.
\end{theorem}

As before, the conjecture has been proven in many other cases already:

\begin{enumerate}
	\item Again, when $n = 1$, the result follows because Conjecture \ref{BJ} is true in this case.
	\item In \cite{BJ} for $p > 2$, $n = 2$ and arbitrary $K$ and $\overline\rho$, the authors apply Nakamura's work in \cite{Nak} to prove Conjecture \ref{DensityConj} as a consequence of their proof of Conjecture \ref{BJ}.
	\item In \cite{CDP} for $p = 2$, $n = 2$, $K = \mathbf{Q}_2$ and $\overline\rho$ the trivial representation, the authors apply Kisin's work in \cite{Kis10} to prove Conjecture \ref{DensityConj} as a consequence of their proof of Conjecture \ref{BJ}.
	\item In \cite[Theorem 4.4]{Nak}, Conjecture \ref{DensityConj} is proven when $H^0(G_K, \ad^0\overline\rho \otimes \epsilon) = 0$ and either $\mu_{p^\infty}(K) = \{1\}$ or $p \nmid n$ (as above we ignore the issue of framing). In fact the theorem also requires the existence of crystalline lifts with regular Hodge-Tate weights, which was recently proven in \cite{EG}.
\end{enumerate}

\subsection{Methods}

First, we give an overview of the proof of Theorem \ref{FirstMainTheorem} and some auxiliary results along the way. Following \cite{CDP} we use the description of the maximal pro-$p$ quotient of $G_K$ given in \cite{Dem} to explicitly realize $R^\square_{\overline\rho}$ as a formal moduli space of tuples of matrices satisfying a certain ``Demu\v{s}kin equation'': this is Corollary \ref{structure}. Using this description and some facts about moduli spaces of matrices proven by Helm in \cite{Helm}, we show the following.

\begin{proposition}[Proposition \ref{dimension}]
The ring $R^\square_{\overline\rho}$ is a complete intersection ring of dimension $1+n^2([K:\mathbf{Q}_p]+1)$, and is $\mathscr{O}_L$-flat. In particular, $R^\square_{\overline\rho}$ is Cohen-Macaulay.
\end{proposition}

Points of $\mathcal{X}_{\overline\rho}^\square$ parametrize characteristic $0$ deformations of $\overline\rho$ valued in finite extensions of $L$. A standard Galois cohomology argument allows us to show that at singular points these deformations are reducible, which allows us to bound the dimension of the singular locus: here we follow the dimension counting technique introduced in \cite[Section 3]{Ger}. Using Serre's criterion for normality, we deduce the following.

\begin{theorem}[Theorem \ref{normal}]\label{main}
$\mathcal{X}^\square_{\overline\rho}$ is normal. In particular, its connected components are irreducible.
\end{theorem}

We then compute the deformation theory of the trivial character $\mathbf{1}: G_K^{\mathrm{ab}} \to k_L^\times$, and show that two closed points in $\mathcal{X}_{\overline\rho}^\square$ whose determinants land in the same connected component of $\mathcal{X}_{\mathbf{1}}^\square$ actually live in the same connected component of $\mathcal{X}^\square_{\overline\rho}$. We show this by reducing to the case $n = 2$, which is handled in \cite{BJ}.

Second, we say a quick word about the proof of Theorem \ref{secondref}. Our proof that the Zariski closure of regular points is a union of some irreducible components of $\mathcal{X}^\square_{\overline\rho}$ is almost exactly the same as Nakamura's proof in \cite{Nak}. His argument is inspired by, and directly generalizes, Chenevier's work in \cite{Che3} and \cite{Che2}, which in turn generalizes Gouv\^{e}a-Mazur's ``infinite fern'' argument that they originally introduced in \cite{Maz2}. However, since Nakamura assumes throughout the paper that $\overline\rho$ has only scalar endomorphisms, we use the trianguline variety as studied in \cite{BHS} in lieu of Nakamura's finite slope subspace, which is defined in terms of framed deformation functors and is thus better suited to our purposes.

We end the introduction with a few remarks.

\begin{remark}\label{caveats}
$\mbox{}$
\begin{enumerate}
	\item B\"ockle and Juschka use a versal hull $R^{\mathrm{ver}}_{\overline\rho}$ of the associated unframed deformation problem instead of $R^\square_{\overline\rho}$, but there is always a formally smooth map $D^\square_{\overline\rho} \to h_{R^{\mathrm{ver}}_{\overline\rho}}$, so for our purposes it suffices to study either ring. Actually, since $\overline\rho$ is trivial, $R^\square_{\overline\rho} = R^{\mathrm{ver}}_{\overline\rho}$: this is Lemma \ref{framedversal}.
	\item Deformations of the trivial $\overline\rho$ factor through the maximal pro-$p$-quotient of $G_K$, which has a nice presentation that we can work with. A general representation $\overline\rho: G_K \to \mathrm{GL}_n(k_L)$ trivializes over some finite extension $K'/K$ because $\mathrm{GL}_n(k_L)$ is discrete, and this induces a map $D^\square_{\overline\rho} \to D^\square_{\overline\rho|_{K'}}$, but we are not sure how much can be said about this map for general $\overline\rho$: this paper only studies the target of the map.
	\item For convenience, we work mostly with the scheme-theoretic generic fiber $X^\square_{\overline\rho} := \Spec R^\square_{\overline\rho}[1/\varpi_L]$. To see that this is enough, note first of all $\mathcal{X}_{\overline\rho}^\square$ is a normal rigid space if and only if $X_{\overline\rho}^\square$ is a normal scheme by \cite[Lemma 7.1.9]{DeJ}. Secondly, the irreducible components of $\mathcal{X}^\square_{\overline\rho}$ are in canonical bijection with the irreducible components of $\Spf R^\square_{\overline\rho}$ by \cite[Theorem 2.3.1]{Con}, i.e. the images of the connected components of a normalization of $\Spf R^\square_{\overline\rho}$. But these are in canonical bijection with the irreducible components of $X^\square_{\overline\rho}$ as explained in Remark \ref{genericf} below.
	\item The assumption that $p > n$ is a hypothesis that we hope to be able to remove; for now, we do not see an easy way to deal with the extra technicalities that arise. We also avoid the case when $p = 2$ and $|\mu_{p^\infty}(K)| = 2$: see the discussion after Corollary \ref{structure} for an explanation.
\end{enumerate}
\end{remark}

\subsection{Acknowledgements}

I would like to thank James Newton for suggesting the problem, and for constant support and guidance. I would also like to thank Dougal Davis and Andrew Graham for helpful conversations, and Toby Gee, Pol van Hoften, and Carl Wang-Erickson for comments on earlier drafts. Finally, I would like to thank the anonymous referee for various useful comments on the first submission.

This work was supported by the Engineering and Physical Sciences Research Council [EP / L015234 / 1], The EPSRC Centre for Doctoral Training in Geometry and Number Theory (The London School of Geometry and Number Theory), King's College London.

\section{Galois Deformation Rings}\label{GaloisDeformationRings}

We now restate the setup. Fix a prime $p$ and a finite extension $K/\mathbf{Q}_p$ of degree $d = [K:\mathbf{Q}_p]$. Fix another finite extension $L/\mathbf{Q}_p$ with ring of integers $\mathscr{O}_L$ and maximal ideal $\mathfrak{m}_L$, and let $k_L = \mathscr{O}_L/\mathfrak{m}_L$ denote the residue field. We always assume $|\mu_{p^\infty}(L)| \geq |\mu_{p^\infty}(K)|$. Fix an integer $n > 1$. Let $\epsilon: G_K \to L^\times$ denote the $p$-adic cyclotomic character.

We are interested in deforming a continuous representation $\overline\rho: G_K \to \mathrm{GL}_n(k_L)$. Let $\mathsf{Art}_{\mathscr{O}_L}$ denote the category whose objects are local Artinian $\mathscr{O}_L$-algebras $A$ together with a surjective reduction map $A \twoheadrightarrow k_L$, and whose morphisms are local $\mathscr{O}_L$-algebra homomorphisms $A \to B$ respecting the reduction maps to $k_L$. Then the framed deformation problem for $\overline\rho$ is the functor
\begin{align*}
	D_{\overline\rho}^\square: \mathsf{Art}_{\mathscr{O}_L} &\to \mathsf{Set} \\
	A &\mapsto \{\text{lifts of } \overline\rho \text{ to continuous representations } G_K \to \mathrm{GL}_n(A)\}
\end{align*}
and the unframed deformation problem is the functor
\begin{align*}
	D_{\overline\rho}: \mathsf{Art}_{\mathscr{O}_L} &\to \mathsf{Set} \\
	A &\mapsto D^\square_{\overline\rho}(A)/(1+\mathrm{Mat}_n(\mathfrak{m}_A))\text{-conjugacy}.
\end{align*}
Note $D_{\overline\rho}$ may not be pro-representable, but it always has a versal hull, i.e. a complete local Noetherian ring $R^{\mathrm{ver}}_{\overline\rho}$ and a formally smooth map $h_{R^{\mathrm{ver}}_{\overline\rho}} \to D_{\overline\rho}$ of deformation functors that induces an isomorphism on tangent spaces.

For any functor $D: \mathsf{Art}_{\mathscr{O}_L} \to \mathsf{Set}$, let $T(D) = D(k_L[x]/x^2)$ denote its tangent space.

\begin{lemma}\label{framedversal}
The forgetful map $D^\square_{\overline\rho} \to D_{\overline\rho}$ factors through a (non-unique) formally smooth morphism $D^\square_{\overline\rho} \to h_{R^{\mathrm{ver}}_{\overline\rho}}$, which is an isomorphism if $\overline\rho$ is the trivial representation.
\end{lemma}
\begin{proof}
Note $D^\square_{\overline\rho} \to D_{\overline\rho}$ is formally smooth, and is therefore a versal deformation. Then \cite[Tag 06T5(3)]{stacks-project} gives us the formally smooth map $D^\square_{\overline\rho} \to h_{R^{\mathrm{ver}}_{\overline\rho}}$. Let $T(D) = D(k_L[x]/x^2)$ denote the tangent space to a deformation functor $D: \mathsf{Art}_{\mathscr{O}_L} \to \mathsf{Set}$. Then a standard computation shows that the induced map of tangent spaces $T(D^\square_{\overline\rho}) \to T(D_{\overline\rho})$ is just the quotient map $Z^1(G_K,\ad\overline\rho) \to H^1(G_K,\ad\overline\rho)$ (where $Z^1$ denotes continuous cocycles and $H^1$ is continuous group cohomology). If $\overline\rho$ is trivial, then the coboundaries $B^1(G_K,\ad\overline\rho) = 0$, so $D^\square_{\overline\rho} \to D_{\overline\rho}$ induces an isomorphism on tangent spaces, and is thus a versal hull.
\end{proof}

\subsection{Presentation of \texorpdfstring{$R^\square_{\overline\rho}$}{R}}

For the rest of this article, take $\overline\rho$ to be the trivial representation.

Since lifts $\rho \in D^\square_{\overline\rho}(A)$ reduce to $\overline\rho$, they must factor through $1 + \mathrm{Mat}_n(\mathfrak{m}_A) \hookrightarrow \mathrm{GL}_n(A)$, where $\mathrm{Mat}_n(\mathfrak{m}_A)$ is the set of $n \times n$ matrices with entries in the maximal ideal $\mathfrak{m}_A \subset A$.

\begin{lemma}\label{pro-p}
For any $A$ in $\mathsf{Art}_{\mathscr{O}_L}$, the group $1 + \mathrm{Mat}_n(\mathfrak{m}_A)$ is a $p$-group.
\end{lemma}
\begin{proof}
It suffices to count the number of elements in $\mathrm{Mat}_n(\mathfrak{m}_A)$, for which it suffices to count the number of elements in $\mathfrak{m}_A$. In the maximal ideal filtration $\mathfrak{m}_A \supset \mathfrak{m}_A^2 \supset \cdots \supset \mathfrak{m}_A^k = 0$, the successive quotients are finite $k_L$-vector spaces. The lemma follows.
\end{proof}

So since any deformation $\rho: G_K \to 1 + \mathrm{Mat}_n(\mathfrak{m}_A)$ is a map from $G_K$ into a $p$-group, it must factor through $G_K \twoheadrightarrow G_K^p$, where $G_K^p$ denotes the maximal pro-$p$-quotient of $G_K$. Let $q = |\mu_{p^\infty}(K)|$. We note the following results of Shafarevich and Demu\v{s}kin:

\begin{theorem}\label{groupstructure}
$\mbox{}$
\begin{enumerate}
	\item (Shafarevich \cite{Sha}) If $q = 1$, then $G_K^p$ is a free pro-$p$-group of rank $d+1$.
	\item (Demu\v{s}kin \cite{Dem}) If $q \geq 3$ then $G_K^p$ is isomorphic to the quotient of the free pro-$p$-group on $d+2$ generators $g_1,\dots,g_{d+2}$ by the relation
	\[ g_1^q[g_1,g_2][g_3,g_4] \cdots [g_{d+1},g_{d+2}] = 1, \]
where $[g,h] = ghg^{-1}h^{-1}$.
\end{enumerate}
\end{theorem}

When $q = 2$, the group $G_K^{(2)}$ is again cut out by one relation, which looks a lot like the one in part 2 of Theorem \ref{groupstructure}. This is due to Serre in \cite{Ser} when $d$ is odd and Labute in \cite{Lab} when $d$ is even. We suspect that one could prove that $X_{\overline\rho}^\square$ is normal in these exceptional cases using the same method, but since our crystalline density result assumes $p > n$, we have decided not to pursue this.

\begin{remark}
In \cite{Dem}, Demu\v{s}kin uses the convention that $[g,h] = g^{-1}h^{-1}gh$ rather than $[g,h] = ghg^{-1}h^{-1}$. However, following Serre's proof in \cite[Section 6]{Ser}, we may use either convention: the basis for $G_K^p$ that we get depends on which convention we use to define the lower exponent-$p$ central series of $G_K^p$ in the proof of part 2 of Theorem \ref{groupstructure}. See also the discussion preceding \cite[Theorem 7.5.14]{Neu}.
\end{remark}

We can use these presentations to determine the representing ring for $D^\square_{\overline\rho}$.

\begin{proposition}\label{representable}
Suppose $G = \left\langle g_1,\dots,g_s : r(g_1,\dots,g_s)=1 \right\rangle$ is the quotient of the free pro-$p$-group on $s$ generators by the relation $r(g_1,\dots,g_s)=1$. Then the framed deformation functor $D_{\overline\rho}^\square$ of the trivial representation $\rho$ is pro-represented by the complete Noetherian local ring
	\[ R^\square_{\overline\rho} = \mathscr{O}_L\llbracket X_1,\dots,X_s \rrbracket/(r(\widetilde{X}_1,\dots,\widetilde{X}_s)-I), \]
where each $X_i$ is an $n \times n$ matrix of indeterminates, and $\widetilde{X}_i := X_i + I$.
\end{proposition}
\begin{proof}
For $A \in \mathsf{Art}_{\mathscr{O}_L}$, a continuous lift $\rho: G \to 1+\mathrm{Mat}_n(\mathfrak{m}_A)$ is determined by where it sends $g_1,\dots,g_{d+2}$. Thus we can define a map
	\[ f_\rho: R^\square_{\overline\rho} \to A, X_i \mapsto \rho(g_i) - I \]
that is continuous because $\rho(g_i) - I \in \mathrm{Mat}_n(\mathfrak{m}_A)$, and well-defined because
	\[ r(\widetilde{X}_1,\dots,\widetilde{X}_s) \mapsto r(\rho(g_1),\dots,\rho(g_s)) = \rho(r(g_1,\dots,g_s)) = I. \]
Conversely, given a continuous map $f: R^\square_{\overline\rho} \to A$, we can define a continuous map $G \to 1+\mathrm{Mat}_n(\mathfrak{m}_A)$ taking $g_i$ to $f(\widetilde{X}_i)$. These are inverse constructions, and give an isomorphism of functors.
\end{proof}

Let $\rho^\square: G_K \to \mathrm{GL}_n(R^\square_{\overline\rho})$ denote the universal representation.

\begin{corollary}\label{structure}
Let $q$ be the largest power of $p$ such that $K$ contains the $q$th roots of unity. Then
	\[ R^\square_{\overline\rho} = \begin{cases} \mathscr{O}_L\llbracket X_1,\dots,X_{d+1} \rrbracket & q=1 \\ \mathscr{O}_L\llbracket X_1,\dots,X_{d+2} \rrbracket/(\widetilde{X}_1^q[\widetilde{X}_1,\widetilde{X}_2]\cdots[\widetilde{X}_{d+1},\widetilde{X}_{d+2}]-I) & q \geq 3 \end{cases} \]
\end{corollary}

We are interested in the irreducible components of $\Spec R^\square_{\overline\rho}[1/\varpi_L]$. If $q=1$, then there is clearly only one such component. Since we have decided not to treat $q = 2$, we assume $q > 2$ in the remainder of this article.

The problem with studying the geometry of $\Spec R^\square_{\overline\rho}$ is that there is a particularly nasty singularity at the unique closed point. A nicer approach is to study the generic fiber, which has lots of closed points admitting a nice moduli description, and whose singularities are far easier to control. There are two (basically equivalent) ways of doing this: one can either study the rigid generic fiber $\mathcal{X}^\square_{\overline\rho}$ in the sense of Berthelot (see e.g. \cite[Section 7]{DeJ}), or one can just study the scheme-theoretic generic fiber
	\[ X^\square_{\overline\rho} := \Spec R^\square_{\overline\rho}[1/\varpi_L]. \]
We will use both perspectives: to study irreducible components, it will suffice to use $X^\square_{\overline\rho}$ (although some of the path-connectedness arguments later on in the paper are best thought of rigid analytically). Later, when showing density of crystalline points, we will use $\mathcal{X}^\square_{\overline\rho}$.

\begin{remark}\label{genericf}
Restricting to the generic fiber still allows us to study irreducible components of $\Spec R^\square_{\overline\rho}$ itself, and thus of $\Spf R^\square_{\overline\rho}$ (taken with respect to the maximal ideal): we will show that $R^\square_{\overline\rho}$ is $\mathscr{O}_L$-flat, which implies that the map $R^\square_{\overline\rho} \to R^\square_{\overline\rho}[1/\varpi_L]$ induces a bijection on irreducible components. To see this, note that there is a bijection between irreducible components of $X^\square_{\overline\rho}$ and irreducible components of $\Spec R^\square_{\overline\rho}$ that have nonempty intersection with $X^\square_{\overline\rho}$. Then note that each irreducible component intersects the generic fiber, since $R^\square_{\overline\rho}$ is $\mathscr{O}_L$-flat. Therefore, it suffices to study the irreducible components of $X^\square_{\overline\rho}$, but in fact we will first show that $X^\square_{\overline\rho}$ is normal, and then just study the connected components.
\end{remark}

\subsection{Dimension of \texorpdfstring{$R^\square_{\overline\rho}$}{R}}

First we note the following fact.

\begin{lemma}\label{pairs}
Let $M(n,q+1)_{\mathscr{O}_L} \subset \mathrm{GL}_{n,\mathscr{O}_L} \times_{\mathscr{O}_L} \mathrm{GL}_{n,\mathscr{O}_L}$ be the closed subspace of pairs of invertible matrices $X,Y$ satisfying $XYX^{-1} = Y^{q+1}$. Then $M(n,q+1)_{\mathscr{O}_L}$ is a local complete intersection, and is Cohen-Macaulay and flat of relative dimension $n^2$ over $\Spec \mathscr{O}_L$.
\end{lemma}
\begin{proof}
The proof is given in the first two paragraphs of \cite[Theorem 2.5]{Sho}, although Shotton attributes the proof to Helm in \cite[Proposition 4.2]{Helm}, who in turn attributes the argument to Choi in \cite{Choi}.
\end{proof}

\begin{proposition}\label{dimension}
The ring $R^\square_{\overline\rho}$ is a complete intersection ring of dimension $1 + n^2(d+1)$, and $\mathscr{O}_L$-flat. In particular $R^\square_{\overline\rho}$ is Cohen-Macaulay.
\end{proposition}
\begin{proof}
Note $\mathscr{O}_L\llbracket X_1,\dots,X_{d+2} \rrbracket$ has dimension $1 + n^2(d+2)$ and we quotient out by $n^2$ equations so we would expect $R^\square_{\overline\rho}$ to have dimension $1+n^2(d+1)$. In fact, if we further quotient out by the $n^2d$ indeterminates defining $X_3,\dots,X_{d+2}$ we are left with
	\[ R' := \mathscr{O}_L\llbracket X_1,X_2 \rrbracket/(\widetilde{X}_1^q[\widetilde{X}_1,\widetilde{X}_2] - I), \]
which can be rewritten as
	\[ \mathscr{O}_L\llbracket X_1,X_2 \rrbracket/(\widetilde{X}_2\widetilde{X}_1\widetilde{X}_2^{-1} - \widetilde{X}_1^{q+1}), \]
whose formal spectrum is just the formal completion of $M(n,q+1)_{\mathscr{O}_L}$ at the closed $k_L$-point $x_0$ defined by $X = Y = I$, which one can see by noting that $M(n,q+1)_{\mathscr{O}_L}^{\wedge_{x_0}}$ represents a functor $\mathsf{Art}_{\mathscr{O}_L} \to \mathsf{Set}$, which is also representable by $R'$. Thus by Lemma \ref{pairs}, $\Spec R'$ is flat of relative dimension $n^2$ over $\Spec \mathscr{O}_L$, and thus has dimension $n^2+1$. Therefore, $\Spec R'/\varpi_L$ has dimension $n^2$. In summary, taking the quotient of $\mathscr{O}_L\llbracket X_1,\dots,X_{d+2} \rrbracket$ by the $n^2(d+1)+1$ equations
	\[ \widetilde{X}_2\widetilde{X}_1\widetilde{X}_2^{-1}-\widetilde{X}_1^{q+1},X_3,\dots,X_{d+2},\varpi_L \]
gives a ring of dimension $n^2$, so these $n^2(d+1)+1$ equations must form a regular sequence (in any order since $\mathscr{O}_L\llbracket X_1,\dots,X_{d+2} \rrbracket$ is local Noetherian) and thus $\dim R^\square_{\overline\rho} = 1 + n^2(d+1)$. We conclude that $R^\square_{\overline\rho}$ is complete intersection and therefore also Cohen-Macaulay. Note $\varpi$ by itself forms a regular sequence so $R^\square_{\overline\rho}$ is $\mathscr{O}_L$-torsion free, and in particular $\mathscr{O}_L$-flat since $\mathscr{O}_L$ is a DVR.
\end{proof}

\section{Normality}

In this section, we show that $X^\square_{\overline\rho}$ is a normal scheme, so that in fact its irreducible components are just the connected components.

For any closed point $x \in X^\square_{\overline\rho}$ we may define the residual representation
	\[ \rho_x: G_K \xrightarrow{\rho^\square} \mathrm{GL}_n(R^\square_{\overline\rho}) \to \mathrm{GL}_n(R^\square_{\overline\rho}[1/\varpi_L]) \to \mathrm{GL}_n(\kappa(x)) \]
where $\kappa(x)$ is the residue field of the stalk at $x \in X^\square_{\overline\rho}$.

\begin{lemma}\label{residue}
If $x \in X^\square_{\overline\rho}$ is a closed point, then the residue field $\kappa(x)$ is a finite extension of $L$, and the image of $\rho_x$ lands in $\mathrm{GL}_n(\mathscr{O}_{\kappa(x)})$, where $\mathscr{O}_{\kappa(x)} \subset \kappa(x)$ is the ring of integers.
\end{lemma}
\begin{proof}
This proof is based on \cite[Lemma 5.1.1]{Bre}. Let $x = \mathfrak{m}$ be some maximal ideal in $R^\square_{\overline\rho}[1/\varpi_L]$, and let $\mathfrak{p}$ denote the corresponding prime ideal in $R^\square_{\overline\rho}$. Note $(R^\square_{\overline\rho}/\mathfrak{p})[1/\varpi_L] \cong R^\square_{\overline\rho}[1/\varpi_L]/\mathfrak{m}$, which is a field, so in particular $\dim((R^\square_{\overline\rho}/\mathfrak{p})[1/\varpi_L]) = 0$. By $\mathscr{O}_L$-flatness, $\varpi_L \in R^\square_{\overline\rho}$ is not nilpotent and $R^\square_{\overline\rho}/\mathfrak{p}$ is a local Noetherian domain, hence equidimensional, so \cite[Lemma 2.3]{CDP} tells us that $\dim(R^\square_{\overline\rho}/\mathfrak{p}) = 1$. Since $\varpi_L \not \in \mathfrak{p}$, it follows that $\dim R^\square_{\overline\rho}/(\varpi_L,\mathfrak{p}) = 0$, i.e.  $R^\square_{\overline\rho}/(\varpi_L,\mathfrak{p})$ is a local Artinian $\mathscr{O}_L$-algebra, hence its underlying set is finite.

Now fix a (finite) set $\widetilde{S}$ consisting of a lift in $R^\square_{\overline\rho}/\mathfrak{p}$ for each residue class in $R^\square_{\overline\rho}/(\varpi_L,\mathfrak{p})$. Then given an element $a_0 \in R^\square_{\overline\rho}/\mathfrak{p}$, by reducing mod $\varpi_L$ we can find $b_0$ generated by elements of $\widetilde{S}$ over $\mathscr{O}_L$ such that
	\[ a_0 - b_0 \in \varpi_L R^\square_{\overline\rho}/\mathfrak{p}. \]
So there is some $a_1 \in R^\square_{\overline\rho}/\mathfrak{p}$ such that $a_0 - b_0 = \varpi_La_1$. Repeating this for $a_1$, then $a_2$, etc, we find after rearranging that
	\[ a_0 = b_0 + \varpi_Lb_1 + \varpi_L^2b_2 + \cdots, \]
which converges, and by rearranging the terms we can express $a_0$ in terms of elements of $\widetilde{S}$ over $\mathscr{O}_L$.

Therefore, $R^\square_{\overline\rho}/\mathfrak{p}$ is a finitely generated $\mathscr{O}_L$-module and thus
	\[ (R^\square_{\overline\rho}/\mathfrak{p})[1/\varpi_L] = \kappa(x), \]
is a finite extension of $L$. Furthermore, the image of $R^\square_{\overline\rho}$ in $\kappa(x)$ lands in $\mathscr{O}_{\kappa(x)}$: this follows from the remarks in \cite[Section 7.1.8]{DeJ}.
\end{proof}

Thus we may write $\rho_x: G_K \to \mathrm{GL}_n(\mathscr{O}_{\kappa(x)})$.

\begin{proposition}\label{reducible}
If $x = \mathfrak{m} \in X^\square_{\overline\rho}$ is a singular (i.e. not regular) closed point, then $\rho_x$ is reducible.
\end{proposition}

\begin{proof}
\cite[Section 2.3]{Kis09} (in particular Lemma 2.3.3 and Proposition 2.3.5) shows that the $\mathfrak{m}$-adic completion $(R^\square_{\overline\rho})_\mathfrak{m}^\wedge$ represents the framed deformation functor $D^\square_{\rho_x}: \mathsf{Art}_{\kappa(x)} \to \mathsf{Set}$. Since $x$ is singular, $(R^\square_{\overline\rho})_\mathfrak{m}^\wedge$ is not formally smooth, and thus the deformation problem for $\rho_x$ is obstructed, i.e. $H^2(G_K,\ad\rho_x) \neq 0$, but by local Tate duality, this is the same as $H^0(G_K,\ad\rho_x \otimes \epsilon) \neq 0$ (note the coadjoint representation is isomorphic to the adjoint representation for $\mathrm{GL}_n$), which can be rewritten as
	\[ \mathrm{Hom}_{\kappa(x)[G_K]}(\rho_x,\rho_x \otimes \epsilon) \neq 0. \]
Thus there exists some nonzero $\kappa(x)[G_K]$-map $\psi: \rho_x \to \rho_x \otimes \epsilon$, which can never be an isomorphism because $\det(\rho_x)$ and $\det(\rho_x \otimes \epsilon) = \epsilon^n \det(\rho_x)$ are non-isomorphic characters, so $\rho_x$ is reducible.
\end{proof}

\subsection{The Reducible Locus}

To control the singular locus, we imitate Geraghty's approach in \cite[Section 3]{Ger} and try to bound the locus of points of $X^\square_{\overline\rho}$ whose corresponding representation is reducible. The only real differences between what we do here and what Geraghty does in his thesis is that we will need to consider flag varieties for every parabolic subgroup (not just the Borel), and we need to parametrize pairs $(\rho,\mathrm{Fil}^\bullet)$ where $\rho$ fixes $\mathrm{Fil}^\bullet$ but does not fix a finer flag (see Remark \ref{irr}).

Pick a $k$-tuple of positive integers $\underline{n} = (n_1,\dots,n_k)$ such that $\sum_{i=1}^k n_i = n$, and let $\mathcal{F}_{\underline{n}} \in \mathsf{Sch}_{\mathscr{O}_L}$ be the flag variety associated to $\underline{n}$, i.e. the scheme representing the functor $\mathcal{F}_{\underline{n}}: \mathsf{Alg}_{\mathscr{O}_L} \to \mathsf{Set}$ defined by
	\[ \mathcal{F}_{\underline{n}}: A \mapsto \begin{Bmatrix} \text{filtrations } 0 \subset \mathrm{Fil}^1 \subset \dots \subset \mathrm{Fil}^k = A^n \text{ by projective} \\ \text{$A$-submodules that are locally direct summands,} \\ \text{and that satisfy rank}_A (\mathrm{Fil}^i) = n_1 + \dots + n_i\end{Bmatrix}. \]
Then $\Spec R^\square_{\overline\rho} \times_{\mathscr{O}_L} \mathcal{F}_{\underline{n}}$ is an $\mathscr{O}_L$-scheme whose $A$-points (for $A \in \mathsf{Alg}_{\mathscr{O}_L}$) are pairs $(f,\mathrm{Fil}^\bullet)$, where $\mathrm{Fil}^\bullet$ is as above, and $f: R^\square_{\overline\rho} \to A$ is an $\mathscr{O}_L$-algebra morphism. Note $f$ induces a representation
	\[ \rho_f: G_K \xrightarrow{\rho^\square} \mathrm{GL}_n(R^\square_{\overline\rho}) \xrightarrow{f} \mathrm{GL}_n(A). \]
Define a subfunctor $\mathcal{G}_{\underline{n}} \hookrightarrow \Spec R^\square_{\overline\rho} \times_{\mathscr{O}_L} \mathcal{F}_{\underline{n}}$ by
	\[ \mathcal{G}_{\underline{n}}(A) = \{(f,\mathrm{Fil}^\bullet) : \text{the action of $G_K$ on $A^n$ via $\rho_f$ preserves $\mathrm{Fil}^\bullet$}\}. \]
\begin{proposition}\label{closedimmersion}
$\mathcal{G}_{\underline{n}}$ is represented by a closed subscheme of $\Spec R^\square_{\overline\rho} \times_{\mathscr{O}_L} \mathcal{F}_{\underline{n}}$.
\end{proposition}
\begin{proof}
It suffices to show that for any $A \in \mathsf{Alg}_{\mathscr{O}_L}$ and any $A$-point of $\Spec R^\square_{\overline\rho} \times_{\mathscr{O}_L} \mathcal{F}_{\underline{n}}$, there exists an ideal $I \subseteq A$ and a map $\mathsf{Alg}_{\mathscr{O}_L}(A/I,-) \to \mathcal{G}_{\underline{n}}$ such that 
\begin{center}\begin{tikzcd}
\mathsf{Alg}_{\mathscr{O}_L}(A/I,-) \dar[hook] \rar & \mathcal{G}_{\underline{n}} \dar[hook] \\
\mathsf{Alg}_{\mathscr{O}_L}(A,-) \rar & \Spec R^\square_{\overline\rho} \times_{\mathscr{O}_L} \mathcal{F}_{\underline{n}}
\end{tikzcd}\end{center}
is Cartesian.

Fix an $A$-point of $\Spec R^\square_{\overline\rho} \times_{\mathscr{O}_L} \mathcal{F}_{\underline{n}}$, which gives a pair $(f,\mathrm{Fil}^\bullet)$ as before. Then since $\mathrm{Fil}^\bullet$ is a filtration of direct summands, we can fix complementary $A$-submodules $N_i \subset A^n$ such that $\mathrm{Fil}^i \oplus N_i = A^n$. These come with surjective projection maps $A^n \xrightarrow{\pi_i} N_i$.

We can now define the ideal $I \subseteq A$ generated by the coefficients of $\pi_i\rho_f(g)v$ with respect to the standard basis of $A^n$ for all $g \in G_K$ and all $v \in \mathrm{Fil}^i$, for each $i = 1,\dots,r$. If we write
	\[ \rho_{f,I}: G_K \xrightarrow{\rho_f} \mathrm{GL}_n(A) \to \mathrm{GL}_n(A/I) \]
then $(R^\square_{\overline\rho} \xrightarrow{f} A \to A/I,\mathrm{Fil}^\bullet \otimes_A A/I)$ is an $A/I$-point of $\mathcal{G}_{\underline{n}}$ (because $\rho_{f,I}$ is now forced to fix $\mathrm{Fil}^\bullet \otimes_A A/I$), which thus gives us the desired map $\mathsf{Alg}_{\mathscr{O}_L}(A/I,-) \to G_{\underline{n}}$. Now given a diagram
\begin{center}\begin{tikzcd}
F \arrow[bend right]{ddr} \arrow[bend left=15]{rrd} \arrow[dashed]{rd}{\exists!} \\
& \mathsf{Alg}_{\mathscr{O}_L}(A/I,-) \dar[hook] \rar & \mathcal{G}_{\underline{n}} \dar[hook] \\
& \mathsf{Alg}_{\mathscr{O}_L}(A,-) \rar & \Spec R^\square_{\overline\rho} \times \mathcal{F}_{\underline{n}}
\end{tikzcd}\end{center}
one checks easily that we get a unique map $F \to \mathsf{Alg}_{\mathscr{O}_L}(A/I,-)$.
\end{proof}

\begin{definition}
For a positive integer $m$, denote by $P(m)$ the finite set of ordered partitions (viewed as ordered tuples) of the integer $m$. If $\underline{m} = (m_1,\dots,m_k)$ is a $k$-tuple of integers for $k \geq 1$, then let $P(\underline{m})$ be the image of the natural concatenation map
	\[ P(m_1) \times \dots \times P(m_k) \to P(m_1 + \dots + m_k)  \]
Finally, let $P(\underline{m})^\circ = P(\underline{m}) \setminus \{\underline{m}\}$ and $P(m)^\circ = P(m) \setminus \{(m)\}$.
\end{definition}

\begin{lemma}
If $\underline{n}' \in P(\underline{n})^\circ$, then the natural map $\mathcal{G}_{\underline{n}'} \to \mathcal{G}_{\underline{n}}$ (taking a filtration of shape $\underline{n}'$ and only remembering that it gives a filtration of shape $\underline{n}$) is proper.
\end{lemma}
\begin{proof}
The partial flag varieties $\mathcal{F}_{\underline{n}'}$ and $\mathcal{F}_{\underline{n}}$ are proper over $\Spec \mathscr{O}_L$, so the map $\mathcal{F}_{\underline{n}'} \to \mathcal{F}_{\underline{n}}$ is proper. We want the top arrow in the following diagram (which is not Cartesian!) to be proper:
\begin{center}\begin{tikzcd}
	\mathcal{G}_{\underline{n}'} \rar \dar[hook] & \mathcal{G}_{\underline{n}} \dar[hook] \\
	\Spec R^\square_{\overline\rho} \times_{\mathscr{O}_L} \mathcal{F}_{\underline{n}'} \rar{\text{proper}} & \Spec R^\square_{\overline\rho} \times_{\mathscr{O}_L} \mathcal{F}_{\underline{n}}
\end{tikzcd}\end{center}
But all of the other arrows are proper, so the top is as well.
\end{proof}

Putting together these maps, we obtain a map
	\[ \bigsqcup_{\underline{n}' \in P(\underline{n})^\circ} \mathcal{G}_{\underline{n}'} \to \mathcal{G}_{\underline{n}}. \]
Since each $\mathcal{G}_{\underline{n}'} \to \mathcal{G}_{\underline{n}}$ is closed (by properness), and the disjoint union is taken over a finite set, the (set-theoretic) image of this map is closed: denote by $\mathcal{G}_{\underline{n}}^{\mathrm{irr}}$ its open complement with the natural subscheme structure.

\begin{remark}\label{irr}
To motivate this definition, note that $\mathcal{G}_{\underline{n}}^{\mathrm{irr}}$ should parametrize pairs $(f,\mathrm{Fil}^\bullet)$ where the induced representation $\rho_f$ fixes $\mathrm{Fil}^\bullet$ but does not fix any finer filtration in $\mathcal{F}_{\underline{n}'}$ for $\underline{n}' \in P(\underline{n})^\circ$ after base changing to an algebraic closure: in fact, $\mathcal{G}_{\underline{n}}^{\mathrm{irr}}$ represents the functor that takes an $\mathscr{O}_L$-algebra $A$ to the set of pairs $(f,\mathrm{Fil}^\bullet)$ such that $\rho_f$ fixes $\mathrm{Fil}^\bullet$ and such that for all geometric points $\overline{s}$ of $\Spec A$, the representation $\rho_{f,\overline{s}}$ does not fix any filtration strictly refining $\mathrm{Fil}^\bullet_{\overline{s}}$.
\end{remark}

Finally, consider the map
	\[ \bigsqcup_{\underline{n} \in P(n)^\circ} \mathcal{G}_{\underline{n}}^{\mathrm{irr}} \to \bigsqcup_{\underline{n} \in P(n)^\circ} \mathcal{G}_{\underline{n}} \to \bigsqcup_{\underline{n} \in P(n)^\circ} (\Spec R_{\overline\rho}^\square \times_{\mathscr{O}_L} \mathcal{F}_{\underline{n}}) \to \Spec R_{\overline\rho}^\square. \]
After passing to the generic fiber, we obtain a map
	\[ \bigsqcup_{\underline{n} \in P(n)^\circ} \mathcal{G}_{\underline{n}}^{\mathrm{irr}}[1/\varpi_L] \to X^\square_{\overline\rho}. \]
The image is closed: to see this note the image is the same as the image of
	\[ \bigsqcup_{\underline{n} \in P(n)^\circ} \mathcal{G}_{\underline{n}}[1/\varpi_L] \to X^\square_{\overline\rho}, \]
(this follows from the moduli description of $\mathcal{G}_{\underline{n}}^{\mathrm{irr}}$ given in Remark \ref{irr}) and then note that $\mathcal{G}_{\underline{n}} \to \Spec R_{\overline\rho}^\square \times_{\mathscr{O}_L} \mathcal{F}_{\underline{n}} \to \Spec R_{\overline\rho}^\square$ is proper. So the scheme-theoretic image is a closed subscheme $X^{\square,\mathrm{red}}_{\overline\rho} \subset X^\square_{\overline\rho}$. Note furthermore, that since $R_{\overline\rho}^\square$ is excellent (as it is complete local Noetherian) the singular locus $X^{\square,\mathrm{sing}}_{\overline\rho} \subset X^\square_{\overline\rho}$ is closed. In fact,
\begin{corollary}
$X^{\square,\mathrm{sing}}_{\overline\rho} \subseteq X^{\square,\mathrm{red}}_{\overline\rho}$.
\end{corollary}
\begin{proof}
\cite[Corollaire 10.5.9 and Proposition 10.3.2]{EGA} imply that $X^\square_{\overline\rho}$ is a Jacobson scheme so it again follows from Proposition 10.3.2 that the closed subset $X^{\square,\mathrm{sing}}_{\overline\rho}$ is Jacobson, which implies that it suffices to show that singular closed points are contained in $X^{\square,\mathrm{red}}_{\overline\rho}$. So pick $x \in X^\square_{\overline\rho}$ that is a singular closed point. By Proposition \ref{reducible}, $\rho_x$ stabilizes a flag $\mathrm{Fil}^\bullet_{\overline{x}} \in \mathcal{F}_{\underline{n}}(\overline{\kappa(x)})$ of some shape determined by $\underline{n} \in P(n)^\circ$ and we can assume that $\underline{n}$ is minimal for this property (extending to a larger algebraically closed field will not affect minimality), so we get a point $(f_{\overline{x}},\mathrm{Fil}^\bullet_{\overline{x}}) \in \mathcal{G}^{\mathrm{irr}}_{\underline{n}}(\overline{\kappa(x)})$, where $f_x$ is the map to the algebraic closure of the residue field. Thus, $\rho_x$ is in the image of the map $\mathcal{G}^{\mathrm{irr}}_{\underline{n}}[1/\pi_L] \to X^\square_{\overline\rho}$.
\end{proof}

Therefore,
	\[ \dim X^{\square,\mathrm{sing}}_{\overline\rho} \leq \dim X^{\square,\mathrm{red}}_{\overline\rho} \leq \max_{\underline{n} \in P(n)^\circ} \dim \mathcal{G}_{\underline{n}}^{\mathrm{irr}}[1/\varpi_L], \]
so in the remainder of this section, we will bound the dimension of each $\mathcal{G}_{\underline{n}}^{\mathrm{irr}}$: later, in Proposition \ref{dimr}, we will see why we need to restrict to $\mathcal{G}_{\underline{n}}^{\mathrm{irr}}$ inside $\mathcal{G}_{\underline{n}}$.

\subsection{Dimension Counting}

To compute the dimension of $\mathcal{G}^{\mathrm{irr}}_{\underline{n}}[1/\varpi_L]$, we can compute
	\[ \max_{x \in \mathcal{G}^{\mathrm{irr}}_{\underline{n}}[1/\varpi_L] \text{ closed}} \dim \mathscr{O}_{\mathcal{G}^{\mathrm{irr}}_{\underline{n}}[1/\varpi_L],x} = \max_{x \in \mathcal{G}^{\mathrm{irr}}_{\underline{n}}[1/\varpi_L] \text{ closed}} \dim \mathscr{O}_{\mathcal{G}_{\underline{n}},x} = \max_{x \in \mathcal{G}^{\mathrm{irr}}_{\underline{n}}[1/\varpi_L] \text{ closed}} \dim \widehat{\mathscr{O}}_{\mathcal{G}_{\underline{n}},x}. \]
We showed in Lemma \ref{residue} that the residue field of a closed point $x \in X^\square_{\overline\rho}$ is a finite extension of $L$. Given a closed point $x \in \mathcal{G}^{\mathrm{irr}}_{\underline{n}}[1/\varpi_L]$, what can we say about its residue field? The map
	\[ \mathcal{G}_{\underline{n}}[1/\varpi_L] \hookrightarrow (\Spec R_{\overline\rho}^\square \times_{\mathscr{O}_L} \mathcal{F}_{\underline{n}})[1/\pi_L] \to X^\square_{\overline\rho} \]
is locally of finite type, so by \cite[Corollaire 10.4.7]{EGA}, the image of $x$ in $X^\square_{\overline\rho}$ is a closed point, and thus we can take the field of definition $F$ of $x$ to be a finite extension of the residue field of its image in $X^\square_{\overline\rho}$, which is in turn a finite extension of $L$. Note
	\[ \mathcal{G}_{\underline{n}} \hookrightarrow \Spec R^\square_{\overline\rho} \times_{\mathscr{O}_L} \mathcal{F}_{\underline{n}} \to \Spec R^\square_{\overline\rho} \]
is proper, so we apply the valuative criterion of properness to the diagram
\begin{center}\begin{tikzcd}
	\Spec F \rar{x} \dar & \mathcal{G}_{\underline{n}} \dar \\
	\Spec \mathscr{O}_F \urar[dashed] \rar & \Spec R^\square_{\overline\rho}
\end{tikzcd}\end{center}
to get a lift $x: \Spec \mathscr{O}_F \to \mathcal{G}_{\underline{n}}$ (by abuse of notation we call both points $x$). But this corresponds to some map $f_x: R^\square_{\overline\rho} \to \mathscr{O}_F$ and some $\mathrm{Fil}^\bullet_x \in \mathcal{F}_{\underline{n}}(\mathscr{O}_F)$. Note $f_x$ induces a representation
	\[ \rho_x: G_K \to \mathrm{GL}_n(\mathscr{O}_F). \]
We also get a map
	\[ \Spec \widehat{\mathscr{O}}_{\mathcal{G}_{\underline{n}},x} \to \Spec \mathscr{O}_{\mathcal{G}_{\underline{n}},x} \to \mathcal{G}_{\underline{n}}, \]
which determines a representation $\widehat{\rho}_x: G_K \to \mathrm{GL}_n(\widehat{\mathscr{O}}_{\mathcal{G}_{\underline{n}},x})$ and a filtration $\widehat{\mathrm{Fil}}^\bullet_x \in \mathcal{F}_{\underline{n}}(\widehat{\mathscr{O}}_{\mathcal{G}_{\underline{n}},x})$.

Let $\mathsf{Art}_F$ denote the category whose objects are local Artinian $F$-algebras $A$ together with a surjective reduction map $A \twoheadrightarrow F$, and whose morphisms are local homomorphisms $A \to B$ respecting the reduction maps to $F$. Now let $D_{\rho_x,\underline{n}}^\square: \mathsf{Art}_F \to \mathsf{Set}$ be the functor taking $B$ to the set of pairs $(\rho,\mathrm{Fil}^\bullet)$ of continuous $\rho: G_K \to \mathrm{GL}_n(B)$ lifting $\rho_x$ and $\mathrm{Fil}^\bullet \in \mathcal{F}_{\underline{n}}(B)$ lifting $\mathrm{Fil}^\bullet_x$ such that $\rho$ preserves $\mathrm{Fil}^\bullet$.
\begin{proposition}\label{localg}
$D_{\rho_x,\underline{n}}^\square$ is pro-representable by $\widehat{\mathscr{O}}_{\mathcal{G}_{\underline{n}},x}$ with the universal representation $\widehat{\rho_x}$ and universal filtration $\widehat{\mathrm{Fil}}^\bullet_x$.
\end{proposition}

Before proving the proposition, we note a lemma.

\begin{lemma}\label{basechange}
If $E/L$ is a finite extension, then the universal framed deformation problem $D^\square_{\overline\rho_E}: \mathsf{Art}_{\mathscr{O}_E} \to \mathsf{Set}$ for the trivial representation $\overline\rho_E: G_K \to \mathrm{GL}_n(k_E)$ is pro-represented by $R^\square_{\overline\rho} \otimes_{\mathscr{O}_L} \mathscr{O}_E$. In particular, a lift of $\overline\rho_E$ to $A \in \mathsf{Art}_{\mathscr{O}_E}$ is given by a unique $\mathscr{O}_E$-algebra map $R^\square_{\overline\rho} \otimes_{\mathscr{O}_L} \mathscr{O}_E \to A$ that descends uniquely to an $\mathscr{O}_L$-algebra map $R^\square_{\overline\rho} \to A$.
\end{lemma}
\begin{proof}
In this case, one can see the first part by, for example, looking at Proposition \ref{representable} and comparing the deformation rings. The second part is clear from the fact that $R^\square_{\overline\rho} \otimes_{\mathscr{O}_L} \mathscr{O}_E \to A$ is an $\mathscr{O}_E$-algebra map.
\end{proof}

\begin{proof}[Proof of Proposition \ref{localg}]
Given an $F$-algebra map $\widehat{\mathscr{O}}_{\mathcal{G}_{\underline{n}},x} \to B$, we can push forward $\widehat{\rho_x}$ and $\widehat{\mathrm{Fil}}^\bullet_x$ to get a pair $(\rho,\mathrm{Fil}^\bullet) \in D_{\rho_x,\underline{n}}^\square(B)$.

Conversely, suppose we are given $(\rho,\mathrm{Fil}^\bullet) \in D^\square_{\rho_x,\underline{n}}(B)$. The idea is to try to use universality of $R^\square_{\overline\rho}$, but $B$ is an $F$-algebra and not a $\mathscr{O}_F$-algebra, so we cannot immediately reduce mod $\mathfrak{m}_F$. But in fact, we can first show that $\rho$ factors through a finitely generated local $\mathscr{O}_F$-subalgebra $A \subset B$ with a surjective map onto $\mathscr{O}_F$: this is exactly Kisin's argument in \cite[Proposition 9.5]{Kis03}.

If we take the composition $G_K \xrightarrow{\rho} \mathrm{GL}_n(A) \to \mathrm{GL}_n(A/\mathfrak{m}_A)$ we get $\rho_x$. But note $R^\square_{\overline\rho} \xrightarrow{f_x} \mathscr{O}_F$ is a local homomorphism, so commutativity of
\begin{center}\begin{tikzcd}
& \mathrm{GL}_n(A) \dar \\
\mathrm{GL}_n(R^\square_{\overline\rho}) \urar \dar \arrow[swap]{r}{\rho_x} & \mathrm{GL}_n(\mathscr{O}_F) \dar \\
\mathrm{GL}_n(k_L) \rar[hook] & \mathrm{GL}_n(k_F)
\end{tikzcd}\end{center}
shows that $\rho$ reduces to the trivial representation valued in $k_F$. By Lemma \ref{basechange}, $\rho$ is induced by a local $\mathscr{O}_L$-algebra map $a: R^\square_{\overline\rho} \to A$.

Suppose we have a different $a': R^\square_{\overline\rho} \to B$ inducing $\rho$. Then $R^\square_{\overline\rho} \xrightarrow{a'} B \to B/\mathfrak{m}_B = F$ is $f_x$ and thus factors through $\mathscr{O}_F$. The aforementioned argument of Kisin in \cite{Kis03} also implies in this case that $a'$ factors through a finitely generated $\mathscr{O}_F$-subalgebra $A' \subset B$, which we can take large enough so that it contains $A$. Then by universality of $R^\square_{\overline\rho} \otimes_{\mathscr{O}_L} \mathscr{O}_F$, we have $a = a'$.

The map $a: R^\square_{\overline\rho} \to B$ specializes to $f_x$ under the reduction map $B \twoheadrightarrow F$. Similarly $\mathrm{Fil}^\bullet$ specializes to $\mathrm{Fil}^\bullet_x$, so in other words, we have constructed a $B$-point of $\mathcal{G}_{\underline{n}}$ that specializes to $x$. Thus, we get a map
	\[ \mathscr{O}_{\mathcal{G}_{\underline{n}},x} \to B \]
that factors through the completion, since $B$ is complete.
\end{proof}

To compute the dimension of $\widehat{\mathscr{O}}_{\mathcal{G}_{\underline{n}},x}$, we can find another object representing $D_{\rho_x,\underline{n}}^\square$ whose dimension can be computed explicitly. Note $D_{\rho_x,\underline{n}}^\square$ contains data about lifting representations, but also data about lifting filtrations, and we can consider these separately, essentially by picking a basis (which we do by fixing a parabolic determined by a fixed element of $\mathcal{F}_{\underline{n}}$).

Let $P$ denote the parabolic subgroup of $\mathrm{GL}_{n,F}$ corresponding to $\mathrm{Fil}^\bullet_x$, so that $\rho_x$ naturally lands in $P(F)$. Let $\mathfrak{p}$ denote the Lie algebra of $P$, which naturally comes equipped with a $G_K$-action via the adjoint action of $P(F)$ on $\mathfrak{p}$: in other words $\sigma \cdot M = \rho_x(\sigma)M\rho_x(\sigma)^{-1}$.

We define a functor
	\[ D_{\rho_x,P}^\square: \mathsf{Art}_F \to \mathsf{Set}, B \mapsto \{\rho: G_K \xrightarrow{\text{cts}} P(B) : \rho \text{ lifts } \rho_x\} \]
parametrizing $P$-deformations of $\rho_x$ to local Artinian $F$-algebras with residue field $F$.
\begin{proposition}\label{dimr}
The deformation problem $D_{\rho_x,P}^\square$ is pro-represented by a complete local Noetherian $F$-algebra $R_{\rho_x,P}^\square$ such that
	\[ \dim R_{\rho_x,P}^\square \leq (d+1)(\dim \mathfrak{p}) + (n^2 - \dim \mathfrak{p}) \]
\end{proposition}
\begin{proof}
The existence of $R_{\rho_x,P}^\square$ is standard, and follows from (for example) a slightly modified version of \cite[Proposition 1.3.1]{Boc1}, replacing $\mathrm{GL}_n$ with $P$. The tangent space of $D_{\rho_x,P}^\square$ consists of the $F[\epsilon]/(\epsilon^2)$-points of $D_{\rho_x,P}^\square$. A standard argument shows that any such lift $\rho \in D_{\rho_x,P}^\square(F[\epsilon]/(\epsilon^2))$ can be written uniquely as $\sigma \mapsto (1+c(\sigma)x)\rho_x(\sigma)$ for some continuous $1$-cocycle $c: G_K \to \mathfrak{p}$. Therefore, the tangent space at the closed point has dimension $\dim Z^1(G_K,\mathfrak{p})$, and another standard argument says that $R_{\rho_x,P}^\square$ can be written as a quotient of a power series ring over $F$ in $\dim Z^1(G_K,\mathfrak{p})$ variables. But we have
\begin{align*}
	\dim Z^1(G_K,\mathfrak{p}) &= \dim H^1(G_K,\mathfrak{p}) + \dim B^1(G_K,\mathfrak{p}) \\
	&= \dim H^1(G_K,\mathfrak{p}) + (\dim\mathfrak{p} - \dim\mathfrak{p}^{G_K}) \\
	&= \dim H^1(G_K,\mathfrak{p}) + \dim\mathfrak{p} - \dim H^0(G_K,\mathfrak{p}) \\
	&= (d+1)\dim \mathfrak{p} + \dim H^2(G_K, \mathfrak{p})
\end{align*}
by the local Euler characteristic formula: here $Z^1$ denotes $1$-cocycles and $B^1$ denotes $1$-coboundaries.

It now suffices to bound $\dim H^2(G_K,\mathfrak{p})$, which by local Tate duality is $\dim H^0(G_K,\mathfrak{p}^\vee \otimes \epsilon)$. If $\mathfrak{n}_P$ is the Lie algebra of the unipotent radical of $P$, then we have short exact sequence of $F[G_K]$-modules
	\[ 0 \to \mathfrak{n}_P \to \mathfrak{p} \to \mathfrak{l}_P \to 0, \]
where $\mathfrak{l}_P$ is the Levi quotient. Dualizing, twisting by the cyclotomic character $\epsilon$, and taking the associated $G_K$-cohomology long exact sequence, we get
	\[ 0 \to H^0(G_K,\mathfrak{l}_P^\vee \otimes \epsilon) \to H^0(G_K,\mathfrak{p}^\vee \otimes \epsilon) \to H^0(G_K,\mathfrak{n}_P^\vee \otimes \epsilon) \to \cdots. \]
Note $\dim H^0(G_K,\mathfrak{n}_P^\vee \otimes \epsilon) \leq \dim \mathfrak{n}_P = n^2 - \dim\mathfrak{p}$, so we are done if we can show that
	\[ H^0(G_K,\mathfrak{l}_P^\vee \otimes \epsilon) = 0. \]
Since $x \in \mathcal{G}_{\underline{n}}^{\mathrm{irr}}(F)$, we can find a basis respecting $\mathrm{Fil}^\bullet_x$ in which
	\[ \rho_x \cong \begin{pmatrix}\alpha_1 & * & * \\ & \ddots & * \\ && \alpha_m \end{pmatrix} \]
for some (absolutely) irreducible representations $\alpha_i$ of dimension $n_i$. One can compute that $\mathfrak{l}_P$ is isomorphic, as a $G_K$-representation, to $\bigoplus_{i=1}^m \mathfrak{g}\mathfrak{l}_{n_i}$, where each $\mathfrak{g}\mathfrak{l}_{n_i}$ is the Lie algebra of $\mathrm{GL}_{n_i}(F)$ equipped with the $G_K$-action induced by $G_K \xrightarrow{\alpha_i} \mathrm{GL}_{n_i}(F) \xrightarrow{\ad} \mathrm{GL}(\mathfrak{g}\mathfrak{l}_{n_i})$. Thus we have an isomorphism of $G_K$-representations
	\[ \mathfrak{l}_P^\vee \otimes \epsilon \cong \bigoplus_{i=1}^m (\mathfrak{g}\mathfrak{l}_{n_i}^\vee \otimes \epsilon), \]
and we conclude that
	\[ H^0(G_K,\mathfrak{l}_P^\vee \otimes \epsilon) = \bigoplus_{i=1}^m H^0(G_K,\mathfrak{g}\mathfrak{l}_{n_i}^\vee \otimes \epsilon) = \bigoplus_{i=1}^m \mathrm{Hom}_{F[G_K]}(\alpha_i,\alpha_i \otimes \epsilon) = 0, \]
where the last equality follows from the fact that the $\alpha_i$ are irreducible and $\alpha_i \not \cong \alpha_i \otimes \epsilon$ (e.g. they have nonisomorphic determinant).
\end{proof}

Now define a functor
	\[ D_{\mathrm{Fil}^\bullet_x}: \mathsf{Art}_F \to \mathsf{Set}, B \mapsto \{\text{lifts of } \mathrm{Fil}^\bullet_x \text{ in } B^n\} \]
This is represented by the completed local ring $\widehat{\mathscr{O}}_{\mathcal{F}_{\underline{n}},\mathrm{Fil}^\bullet_x}$, which is isomorphic to a power series ring over $F$ in $n^2-\dim\mathfrak{p}$ variables (the flag variety is smooth and isomorphic to $\mathrm{GL}_{n,F}/P$). Let $\widehat{\mathrm{Fil}}^\bullet_x \in D_{\mathrm{Fil}_x^\bullet}(\widehat{\mathscr{O}}_{\mathcal{F}_{\underline{n}},\mathrm{Fil}^\bullet_x})$ denote the universal filtration for this deformation problem. Each $\widehat{\mathrm{Fil}}^i_x$ is free since $\widehat{\mathscr{O}}_{\mathcal{F}_{\underline{n}},\mathrm{Fil}^\bullet_x}$ is local, and moreover there exists some $\varphi \in \mathrm{GL}_n(\widehat{\mathscr{O}}_{\mathcal{F}_{\underline{n}},\mathrm{Fil}^\bullet_x})$ such that
	\[ \widehat{\mathrm{Fil}}^\bullet_x = \varphi(\mathrm{Fil}^\bullet_x \otimes_F \widehat{\mathscr{O}}_{\mathcal{F}_{\underline{n}},\mathrm{Fil}^\bullet_x}), \]
and such that $\varphi$ reduces to the identity map mod the maximal ideal of $\widehat{\mathscr{O}}_{\mathcal{F}_{\underline{n}},\mathrm{Fil}^\bullet_x}$.

\begin{proposition}
There is an isomorphism of functors
	\[ D_{\rho_x,\underline{n}}^\square = D_{\rho_x,P}^\square \times D_{\mathrm{Fil}^\bullet_x}. \]
\end{proposition}
\begin{proof}
Fix a point $(\rho,\mathrm{Fil}^\bullet) \in D_{\rho_x,\underline{n}}^\square(A)$ lifting $(\rho_x,\mathrm{Fil}^\bullet_x)$ so that $\rho$ fixes $\mathrm{Fil}^\bullet$. The filtration $\mathrm{Fil}^\bullet$ is induced by a map $\widehat{\mathscr{O}}_{\mathcal{F}_{\underline{n}},\mathrm{Fil}^\bullet_x} \to A$, and if we push forward $\varphi$ along this map then we get some $\varphi_A \in \mathrm{GL}_n(A)$ such that $\mathrm{Fil}^\bullet = \varphi_A(\mathrm{Fil}^\bullet_x \otimes_F A)$, and such that $\varphi_A$ reduces to $1$ mod $\mathfrak{m}_A$. Therefore, $\varphi_A^{-1}\rho\varphi_A$ fixes $\mathrm{Fil}^\bullet_x$ and lands in $P(A)$. Then the map $D^\square_{\overline\rho_x,\underline{n}} \to D^\square_{\overline\rho_x,P} \times D_{\mathrm{Fil}^\bullet_x}$ given by $(\rho,\mathrm{Fil}^\bullet) \mapsto (\varphi_A^{-1}\rho\varphi_A,\mathrm{Fil}^\bullet)$ is a functorial bijection, with inverse $(\rho,\mathrm{Fil}^\bullet) \mapsto (\varphi_A\rho\varphi_A^{-1},\mathrm{Fil}^\bullet)$.
\end{proof}

\begin{corollary}\label{dimg}
The ring $\widehat{\mathscr{O}}_{\mathcal{G}_{\underline{n}},x}$ is isomorphic to a power series ring over $R_{\rho_x,P}^\square$
in $n^2 - \dim \mathfrak{p}$ variables.
\end{corollary}

Now we can simply compute. Recall that we assumed $x \in \mathcal{G}_{\underline{n}}^{\mathrm{irr}}$.

\begin{proposition}
The ring $\widehat{\mathscr{O}}_{\mathcal{G}_{\underline{n}},x}$ satisfies
	\[ \dim \widehat{\mathscr{O}}_{\mathcal{G}_{\underline{n}},x} \leq (d+1)(\dim\mathfrak{p}) + 2(n^2-\dim\mathfrak{p}). \]
\end{proposition}
\begin{proof}
This follows from Proposition \ref{dimr} and Corollary \ref{dimg}.
\end{proof}

The upshot is the following theorem.

\begin{theorem}\label{normal}
Assume $q > 2$. Then $X^\square_{\overline\rho}$ is regular in codimension $1$. In particular, since it is Cohen-Macaulay, $X^\square_{\overline\rho}$ is normal (by Serre's criterion for normality).
\end{theorem}
\begin{proof}
We have shown that
	\[ \dim X^{\square,\mathrm{sing}}_{\overline\rho} \leq \max_{\underline{n} \in P(n)^\circ} \dim \mathcal{G}_{\underline{n}}^{\mathrm{irr}} \leq \max_{\underline{n} \in P(n)^\circ} (d+1)(\dim\mathfrak{p}_{\underline{n}}) + 2(n^2 - \dim \mathfrak{p}_{\underline{n}}). \]
where $\mathfrak{p}_{\underline{n}}$ is the Lie algebra of a parabolic $P_{\underline{n}}$ of shape determined by $\underline{n}$. But note that $\dim X^\square_{\overline\rho} = n^2(d+1)$ by Proposition \ref{dimension}, so the singular locus has codimension
	\[ \dim X^\square_{\overline\rho} - \dim X^{\square,\mathrm{sing}}_{\overline\rho} \geq \min_{\underline{n} \in P(n)^\circ} (d-1)(n^2 - \dim\mathfrak{p}_{\underline{n}}) \]
which is strictly bigger than 1 whenever $d > 2$. If our bound is exactly 1, then we must have $d=2$. This means that either $p = 2$ and $K = \mathbf{Q}_2(\mu_4)$, or $p = 3$ and $K = \mathbf{Q}_3(\mu_3)$. In either case we must have $n = 2$.

One can still prove the theorem in these cases, but it requires a slight modification, so we say a few words about how to do this. By \cite[Proposition 4.2]{CDP}, if $x \in X^\square_{\overline\rho}(F)$ is a singular closed point, there is an exact sequence $0 \to \delta \to \rho_x \to \delta \otimes \epsilon \to 0$. In other words, $\rho_x$ fixes some full flag $0 \subset \mathrm{Fil}^1 \subset \mathrm{Fil}^2 = F^2$, and $\rho_x|_{\mathrm{Fil}^1} \otimes \epsilon = \rho_x|_{\mathrm{Fil}^2/\mathrm{Fil}^1}$. So we can leverage this to get a sharper bound on $\dim H^2(G_K,\mathfrak{p})$ in Proposition \ref{dimr}.

First of all, proper parabolics are Borels, so
	\[ \bigsqcup_{\underline{n} \in P(2)^\circ} \mathcal{G}_{\underline{n}}^{\mathrm{irr}} = \mathcal{G}_{(1,1)}. \]
We can define a closed subspace $\mathcal{G}_{(1,1)}^\epsilon \subset \mathcal{G}_{(1,1)}$.
	\[ \mathcal{G}_{(1,1)}^\epsilon(A) = \{(f,\mathrm{Fil}^\bullet) \in \mathcal{G}_{(1,1)} : \rho_f|_{\mathrm{Fil}^1} = \rho_f|_{\mathrm{Fil}^2/\mathrm{Fil}^1} \otimes \epsilon\}. \]
By the discussion above, the scheme theoretic image of $\mathcal{G}_{(1,1)}^\epsilon[1/\varpi_L] \to X^\square_{\overline\rho}$ still contains $X^{\square,\mathrm{sing}}_{\overline\rho}$. Then if $B$ is the Borel corresponding to $\mathrm{Fil}^\bullet$, one can do a direct matrix computation and show that $H^0(G_K,\mathfrak{n}_B^\vee \otimes \epsilon) = 0$, which implies that $H^2(G_K,\mathfrak{b}) = 0$ where $\mathfrak{b}$ is the Lie algebra of $B$. Thus $\dim X^\square_{\overline\rho} = 12$, but the singular locus is at most $10$-dimensional.
\end{proof}

\section{Irreducible Components}

The goal of this section is to prove that the irreducible components of $X^\square_{\overline\rho}$ are exactly parametrized by the $q$th roots of unity.

To see why this might be true, notice that the equation
\begin{equation}\label{defining} \widetilde{X}_1^q[\widetilde{X}_1,\widetilde{X}_2] \cdots [\widetilde{X}_{d+1},\widetilde{X}_{d+2}] = I \end{equation}
implies that $\det(\widetilde{X}_1)^q = 1$ in $R^\square_{\overline\rho}[1/\varpi_L]$ and thus induces a map
	\[ \pi: X^\square_{\overline\rho} \to \mu_{q,L} := \Spec L[x]/(x^q-1). \]
Recall $|\mu_{p^\infty}(L)| \geq q$, so $\mu_{q,L}$ is just the disjoint union of $q$ copies of $\Spec L$, one for each $q$th root of unity in $L$. Since the image of any connected component of $X_{\overline\rho}^\square$ is connected, it must be sent to one of these points. Thus, for each point $\zeta \in \mu_{q,L}$, $X_\zeta := X_{\overline\rho}^\square \times_{\mu_{q,L}, \zeta} \Spec L$ is a union of connected components of $X^\square_{\overline\rho}$. In fact, we will show that Equation (\ref{defining}) provides enough leverage to connect every pair of points in $X_\zeta$, and thus to conclude that $X_\zeta$ is connected. Thus, the goal is to prove:

\begin{theorem}\label{maintheorem}
Under the assumptions of Theorem \ref{secondref}, the space $X^\square_{\overline\rho}$ breaks into the union of $q$ connected (and therefore irreducible, since $X^\square_{\overline\rho}$ is normal) components
	\[ X^\square_{\overline\rho} = \bigsqcup_{\zeta\in\mu_q(L)} X_\zeta. \]
\end{theorem}

\begin{remark}
As mentioned in the introduction, the assumption that $p > n$ is a hypothesis that we hope to be able to remove; for now, we do not see an easy way to deal with the extra technicalities that arise.
\end{remark}

\subsection{Connectedness}

Now fix a $q$th root of unity $\zeta \in L$. In \cite{CDP}, the notion of arc-connectedness between points is introduced to prove that a scheme is connected. The idea is that for any two closed points $x_0,x_1 \in X_\zeta(F)$ (where $F/L$ is a finite extension) one should find a path from $x_0$ to $x_1$ by exhibiting $x_0 = x(0)$ and $x_1 = x(1)$ for some point $x \in X_\zeta(T_F)$, where $T_F$ is the Tate algebra in one variable over $F$. From the rigid analytic viewpoint, this amounts to connecting $x_0$ and $x_1$ via a path parametrized by a closed unit disk.

More generally, suppose $x_0,x_1 \in X_\zeta(F)$ are two closed points. If we can find a connected $L$-scheme $Y$, and a map $Y \to X_\zeta$ whose image contains both $x_0$ and $x_1$ then they are contained in the same connected component of $X_\zeta$.

There are two examples of $Y$ (as in the definition) that we will use.

\begin{itemize}
	\item In \cite{CDP} the following example is used: let $T_F$ denote the Tate algebra in one variable, which is the subring of $F\llbracket t \rrbracket$ consisting of power series $\sum_n c_nt^n$ for which $|c_n| \to 0$ as $n \to \infty$, equipped with the sup norm taken over the coefficients. We then take $Y = \Spec T_F$. As mentioned before, this is essentially the method of connecting the two points in the associated rigid analytic space attached to $X_\zeta$ via the rigid closed unit disk.

	\item Consider the following affinoid version of $\mathrm{GL}_n$:
		\[ \mathscr{O}_{\mathrm{GL}_n,F} = (\mathscr{O}_F[(x_{ij})_{i,j=1,\dots,n},b]^\wedge_{\mathfrak{m}_F}/(\det(x_{ij})b-1))[1/\varpi_F]. \]
	We then take $Y = \Spec \mathscr{O}_{\mathrm{GL}_n,F}$. Note that $Y(F) = \mathrm{GL}_n(\mathscr{O}_F)$ and thus the $\mathrm{GL}_n(\mathscr{O}_F)$-orbit of an $F$-valued deformation (which is essentially given by some tuple of invertible matrices and here $\mathrm{GL}_n(\mathscr{O}_F)$ acts by conjugating such a tuple) is contained in a single connected component.
\end{itemize}

The rest of this section will be devoted to proving the following Proposition.

\begin{proposition}\label{connected}
For any finite extension $F/L$, any two closed points in $X_\zeta(F)$ are contained in the same connected component of $X_\zeta$.
\end{proposition}

We note the following corollary:

\begin{proof}[Proof of Theorem \ref{maintheorem}]
If $X_\zeta$ is not connected, then since $X_\zeta$ is Jacobson, we can find two \textit{closed} points $x,y \in X_\zeta(F)$ living on different connected components, where $F/L$ can be taken to be finite by Lemma \ref{residue}. The result then follows immediately from Proposition \ref{connected}.
\end{proof}

\subsection{Restriction to a Closed Subspace}

For $F/L$ a finite extension, an $F$-point of $X_\zeta$ is the data of a tuple $(M_1,\dots,M_{d+2}) \subset 1+\mathrm{Mat}_n(\mathfrak{m}_F)$ satisfying the equations
	\[ M_1^q[M_1,M_2]\cdots[M_{d+1},M_{d+2}] = I, \det(M_1) = \zeta. \]
It is difficult to get any useful intuition for this equation in its full form, so it is helpful to first restrict to a certain nicely chosen closed subspace of $X_\zeta$ whose $F$-points satisfy a more useful equation. To see that this suffices, we note the following fact.

\begin{proposition}[{{\cite[Proposition 5.1]{CDP}}}]\label{closedsubspace}
Suppose $A$ is a Cohen-Macaulay Noetherian local ring and $x_1,\dots,x_k,x$ is a regular sequence in $A$. Then every irreducible component of $\Spec A[1/x]$ meets the closed subset $\Spec A/(x_1,\dots,x_k)[1/x]$.
\end{proposition}

In the proof of Proposition \ref{dimension} we showed that the coefficients of $(X_3,\dots,X_{d+2})$ (in any order) along with $\varpi_L$ form a regular sequence in $R^\square_{\overline\rho}$, so in particular, we want to consider the closed subspace $V \subset X^\square_{\overline\rho}$ defined by
	\[ V = \Spec R^\square_{\overline\rho}/(X_3,\dots,X_{d+2})[1/\varpi_L]. \]
Let $V_\zeta = X_\zeta \times_{X^\square_{\overline\rho}} V$.

\begin{corollary}\label{justconnectV}
Fix a finite extension $F/L$. If any two closed points in $V_\zeta(F)$ are contained in the same connected component of $X_\zeta$, then any two closed points in $X_\zeta(F)$ are contained in the same connected component of $X_\zeta$.
\end{corollary}
\begin{proof}
This is implied by Proposition \ref{closedsubspace}, as follows: Let $x_0,x_1$ denote two closed points in $X_\zeta(F)$ and let $Z_0$ and $Z_1$ denote the connected component of $X_\zeta$ containing them, respectively. Note these are also connected components of $X^\square_{\overline\rho}$. By normality of $X^\square_{\overline\rho}$ and Proposition \ref{closedsubspace}, $Z_0$ and $Z_1$ both meet $V \cap X_\zeta = V_\zeta$. Note $Z_0$ and $Z_1$ are both connected, so we can replace $x_0,x_1$ with some closed (note $V_\zeta$ is Jacobson) points $v_0,v_1 \in V_\zeta(F)$, after possibly extending $F$ by a finite extension. But these are contained in the same connected component by assumption.
\end{proof}

In the remainder of the section, we show that any two points in $V_\zeta(F)$ are contained in the same connected component of $X_\zeta$. To do this, we connect every point in $V_\zeta(F)$ to the point corresponding to
	\[ M_1 = \mathrm{diag}(\zeta,1,\dots,1), M_2 = \cdots = M_{d+2} = I. \]

\subsection{Constructing Paths}\label{paths}

Fix a point $x \in V_\zeta(F)$ for some finite extension $F/L$. This is the same as giving a pair $M_1,M_2 \in 1+\mathrm{Mat}_n(\mathfrak{m}_F)$ such that
	\[ M_2M_1M_2^{-1} = M_1^{q+1}. \]
(all other $M_i$ are equal to $I$ in this subspace). Enlarge $F$ if needed so that $F$ contains every eigenvalue of $M_1$. The equation above implies that the $(q+1)$-power map from the set of eigenvalues of $M_1$ to itself is a bijection, so if $\lambda$ is an eigenvalue of $M_1$, then there exists some $a$ in the range $[1,n]$ such that $\lambda^{(q+1)^a} = \lambda$. In other words, $\lambda^{q(qm + a)} = 1$ for some positive integer $m$. So $\lambda^q$ is a $(qm+a)$th root of unity that reduces to $1$ mod $\mathfrak{m}_F$, but $p > n$ and thus $p \nmid a$, so actually $\lambda^q = 1$ by Hensel's Lemma and the fact that $\gcd(qm+a,p) = 1$.

\begin{remark}
It is the possibility of the eigenvalues being higher order $p$-power roots of unity that forces us, for the time being, to use the assumption that $p > n$. We hope to be able to find another argument that works for higher order eigenvalues of $M_1$.
\end{remark}

\begin{proposition}\label{connecttodiagonal}
The point $x \in V_\zeta(F)$ is in the same connected component of $X_\zeta$ as the point defined by $(\mathrm{diag}(\lambda_1,\dots,\lambda_n),I)$, where the $\lambda_i$ is some ordering of the eigenvalues of $M_1$.
\end{proposition}
\begin{proof}
We regard $M_1,M_2$ as elements in $\mathrm{GL}(F^n)$. Since the eigenvalues of $M_1$ are contained in $\mu_q(F)$, we have a decomposition into generalized eigenspaces
	\[ F^n = \bigoplus_{\lambda \in \mu_q(F)} W_\lambda. \]
For any eigenvalue $\lambda \in \mu_q(F)$, consider the filtration $\mathrm{Fil}^\bullet_\lambda$ defined by
	\[ 0 \subset \mathrm{Fil}^1_\lambda = \ker(X_1-\lambda)^1 \subset \cdots \subset \mathrm{Fil}_\lambda^{m_\lambda} = \ker(X_1-\lambda)^{m_\lambda} = W_\lambda, \]
where $m_\lambda$ is the maximum size of a Jordan block of $M_1$ with eigenvalue $\lambda$. Let
	\[ f(x) = \frac{x^{q+1}-\lambda}{x-\lambda} = \prod_{i=1}^q (x-\zeta_{q+1}^i\lambda) \]
where $\zeta_{q+1}$ is a primitive $(q+1)$th root of unity in $F$ (enlarge $F$ if needed). Note $\zeta_{q+1}^i = 1$ if and only if $\zeta_{q+1}^i \equiv 1 \mod \mathfrak{m}_F$ by Hensel's Lemma. Then $M_1^{q+1} - \lambda = f(M_1)(M_1-\lambda)$, and each of the $M_1-\zeta_{q+1}^i\lambda$ are invertible because $\zeta_{q+1}^i\lambda \not\equiv 1 \mod \mathfrak{m}_F$, so $f(M_1)$ is invertible. If $v \in \ker(M_1-\lambda)^a$ then
	\[ f(M_1)^a(M_1-\lambda)^aM_2v = (M_1^{q+1}-\lambda)^aM_2v = M_2(M_1-\lambda)^av = 0, \]
so multiplying by $f(M_1)^{-a}$ shows that $M_2$ preserves both $W_\lambda$ and the filtration $\mathrm{Fil}_\lambda$. Let $n_\lambda = \dim W_\lambda$ and let $e_1,\dots,e_{n_\lambda}$ be a basis of $W_\lambda$ constructed by picking a basis for each $\mathrm{Fil}^i_\lambda$ for which $M_2$ is upper triangular, and then concatenating them together in order. Note $M_1$ clearly respects $\mathrm{Fil}_\lambda$ and acts via the scalar $\lambda$ on $\mathrm{Fil}^i_\lambda/\mathrm{Fil}^{i-1}_\lambda$ so in particular $M_1$ is upper triangular for $e_1,\dots,e_{n_\lambda}$. Thus if $E \in \mathrm{GL}_n(F)$ is the change of basis matrix that takes the standard basis of $F^n$ into the basis determined by the $e_i$, then $EM_1E^{-1}$ and $EM_2E^{-1}$ are both upper triangular. But by the Iwasawa decomposition we can write $E = NE_0$ where $N \in B(F)$ is an element of the standard Borel and $E_0 \in \mathrm{GL}_n(\mathscr{O}_F)$, and thus $M_1' := E_0M_1E_0^{-1}$ and $M_2' := E_0M_2E_0^{-1}$ are still upper triangular. But this allows us to define a map
	\[ \Spec \mathscr{O}_{\mathrm{GL}_n,F} \to V_\zeta \]
taking $g \mapsto (gM_1g^{-1},gM_2g^{-1})$. Specializing at $g = I$ gives $x$, and specializing at $g = E_0 \in \mathrm{GL}_n(\mathscr{O}_F)$ gives a point $x' \in V_\zeta(F)$ determined by $M_1'$ and $M_2'$, so both $(M_1,M_2)$ and $(M_1',M_2')$ live in the same connected component of $V_\zeta$. Now we define a path
	\[ \Spec T_F \to V_\zeta \]
by $t \mapsto (g(t)M_1'g(t)^{-1},g(t)M_2'g(t)^{-1})$ where
	\[ g(t) = \mathrm{diag}(t^{n-1},t^{n-2},\dots,t,1). \]
This is well defined since $M_1',M_2'$ are upper-triangular (in particular, the $g(t)$-conjugated matrix has no negative powers of $t$ and doesn't affect the diagonal). and has the effect of ``killing the strictly upper-triangular part'', in the sense that specializing at $t = 1$ gives the point $x'$ and specializing at $t = 0$ gives the point $x^\star$ with corresponding matrices
	\[ M_1^\star = \mathrm{diag}(\lambda_1,\dots,\lambda_n) \mathrm{ and } M_2^\star = \mathrm{diag}(1+m_1,\dots,1+m_n) \]
for some $m_i \in \mathfrak{m}_F$. But then the path $M_1^\star(t) = M_1^\star$ and $M_2^\star(t) = \mathrm{diag}(1+tm_1,\dots,1+tm_n)$ connects $M_2^\star$ to $I$.
\end{proof}

\begin{proposition}\label{lastconnection}
If $F/L$ is a finite extension and $x \in X_\zeta(F)$ is the point
	\[ M_1 = \mathrm{diag}(\lambda_1,\dots,\lambda_n), M_2 = \dots = M_{d+2} = I \]
with $\lambda_1\cdots\lambda_n=\zeta$ then $x$ is in the same connected component of $X_\zeta$ as the point $x_\zeta \in X_\zeta(F)$ corresponding to
	\[ M_1 = \mathrm{diag}(\zeta,1,\dots,1), M_2 = \dots = M_{d+2} = I. \]
\end{proposition}
\begin{proof}
The idea is to treat the $n = 2$ case and then focus on the $2 \times 2$ diagonal blocks and replace $M_1 = \mathrm{diag}(\lambda_1,\dots,\lambda_n)$ by $\mathrm{diag}(\lambda_1,\dots,\lambda_{n-1}\lambda_n,1)$ and then $\mathrm{diag}(\lambda_1,\dots,\lambda_{n-2}\lambda_{n-1}\lambda_n,1,1)$
etc.

For each $\lambda \in \mu_q(F)$, define a character
\begin{align*}
	c_\lambda: G_K \to (G_K^{p})^{\mathrm{ab}} \cong \langle x_1,\dots,x_{d+2} : x_1^q = 1 \rangle^{\mathrm{ab}} &\to \mathscr{O}_L^\times \\
	(x_1, x_2,\dots,x_{d+2}) &\mapsto (\lambda,1,\dots,1)
\end{align*}
and let $R_{\overline\rho,2}^{\square,\lambda}$ denote the 2-dimensional universal framed deformation ring for $\overline\rho$ with fixed determinant $c_{\lambda}$, which comes with a universal lift $\rho^{\lambda}: G_K \to \mathrm{GL}_2(R_{\overline\rho,2}^{\square,\lambda})$ and generic fiber $X_{\overline\rho,2}^{\square,\lambda} = \Spec R_{\overline\rho,2}^{\square,\lambda}[1/p]$.

By universality, the representation
	\[ c_{\lambda_1} \oplus \cdots \oplus c_{\lambda_{n-2}} \oplus \rho^{\lambda_{n-1}\lambda_n}: G_K \to \mathrm{GL}_n(R_{\overline\rho,2}^{\square,\lambda_{n-1}\lambda_n}) \]
induces a map $R_{\overline\rho}^\square \to R_{\overline\rho,2}^{\square,\lambda_{n-1}\lambda_n}$ and thus a map $X_{\overline\rho,2}^{\square,\lambda_{n-1}\lambda_n} \to X_{\overline\rho}^\square$ which, by definition, descends to a map $X_{\overline\rho,2}^{\square,\lambda_{n-1}\lambda_n} \to X_{\zeta}$. By \cite[Theorem 1.5 and Remark 1.7]{BJ} (combined with the fact that ``versal'' and ``framed'' are interchangeable for the trivial representation, cf. Lemma \ref{framedversal}), $R_{\overline\rho,2}^{\square,\lambda_{n-1}\lambda_n}$ is an integral domain, so $X_{\overline\rho,2}^{\square,\lambda_{n-1}\lambda_n}$ is irreducible. But this $(X_{\overline\rho,2}^{\square,\lambda_{n-1}\lambda_n})$-point of $X_\zeta$ specializes to both $x$ and the point given by
	\[ M_1 = \mathrm{diag}(\lambda_1,\dots,\lambda_{n-1}\lambda_n,1), M_2 = \dots = M_{d+2} = I \]
so these points are contained in the same irreducible component of $X_\zeta$. Now we repeat the process with $c_{\lambda_1} \oplus \cdots \oplus c_{\lambda_{n-3}} \oplus \rho^{\lambda_{n-2}\lambda_{n-1}\lambda_n} \oplus 1$, then $c_{\lambda_1} \oplus \cdots \oplus c_{\lambda_{n-4}} \oplus \rho^{\lambda_{n-3}\lambda_{n-2}\lambda_{n-1}\lambda_n} \oplus 1 \oplus 1$, etc. After $n-1$ iterations of this process, we specialize to $x_\zeta$.
\end{proof}

\begin{remark}
In a previous version of this paper \cite{Iye}, we constructed an explicit path between the two points, which worked under the hypothesis $[K:\mathbf{Q}_p] \geq 6$. At the suggestion of the anonymous referee, we used the results from \cite{BJ} to remove this extra assumption. We thank the referee for this suggestion.
\end{remark}

\begin{proof}[Proof of Proposition \ref{connected}]
By Propositions \ref{connecttodiagonal} and \ref{lastconnection}, any two points in $V_\zeta(F)$ are contained in the same connected component of $X_\zeta$. Then apply Corollary \ref{justconnectV}.
\end{proof}

\subsection{Deforming \texorpdfstring{$\mathbf{1}$}{1}}\label{deforming1}

To prove the main theorem, we need to describe $X_{\det\overline\rho}^\square = X_\mathbf{1}^\square$.

Let $\mathbf{1}: G_K^{\mathrm{ab}} \to k_L^\times$ denote the trivial character, and let $D_{\mathbf{1}}: \mathsf{Art}_{\mathscr{O}_L} \to \mathsf{Set}$ denote the universal deformation problem for $\mathbf{1}$. Deformations of the trivial character to a ring $A \in \mathsf{Art}_{\mathscr{O}_L}$ are valued in the $p$-group $1+\mathfrak{m}_A$, so they are really representations of $(G_K^p)^{\mathrm{ab}}$, which by part 2 of Theorem \ref{groupstructure} is the free abelian pro-$p$ group on generators $g_1,\dots,g_{d+2}$ subject to the single relation $g_1^q = 1$. Therefore, Proposition \ref{representable} (with $n=1$) implies the following lemma.

\begin{lemma}
$D_\mathbf{1}$ is pro-represented by the complete local Noetherian ring
	\[ R_\mathbf{1} = \mathscr{O}_L\llbracket x_1,\dots,x_{d+2} \rrbracket/((1+x_1)^q-1). \]
\end{lemma}

\begin{corollary}
$R_{\mathbf{1}}$ is $\mathscr{O}_L$-torsion free and has $q$ irreducible components, given by $x_1 = \zeta - 1$ for each $\zeta \in \mu_q(L)$. These are also the connected components.
\end{corollary}
\begin{proof}
Note $\varpi_L$ does not divide $(1+x_1)^q-1$, so $R_\mathbf{1}$ is $\mathscr{O}_L$-torsion free. We have
	\[ \mathscr{O}_L\llbracket x_1,\dots,x_{d+2} \rrbracket/((1+x_1)^q-1) \cong \mathscr{O}_L[x_1]/((1+x_1)^q-1) \otimes_{\mathscr{O}_L} \mathscr{O}_L\llbracket x_2,\dots,x_{d+2} \rrbracket, \]
but $\mathscr{O}_L\llbracket x_2,\dots,x_{d+2} \rrbracket$ is an integral domain, so it suffices to describe the irreducible components of $\mathscr{O}_L[x_1]/((1+x_1)^q-1)$. Furthermore, since $R_\mathbf{1}$ is $\mathscr{O}_L$-torsion free it is $\mathscr{O}_L$-flat, so it suffices to check the description of the irreducible components on the generic fiber $L[x_1]/((1+x_1)^q-1)$. But Spec of this ring is just the finite set of $L$-points $x_1 = \zeta-1$ for each $\zeta \in \mu_q(L)$.
\end{proof}

\begin{remark}\label{lubintate}
The local reciprocity map $\widehat{K^\times} \xrightarrow\sim G_K^{\mathrm{ab}} \twoheadrightarrow (G_K^p)^{\mathrm{ab}}$ sends some primitive $q$th root of unity $\zeta$ to $g_1$. The irreducible component containing a closed point $x \in X_\mathbf{1}$ with residue field $F/L$ and corresponding character $\chi_x: G_K^{\mathrm{ab}} \to R_\mathbf{1}^\times \to F^\times$ is determined by the element $\chi_x(g_1) = (\chi_x \circ \mathrm{rec}_K)(\zeta) \in \mu_q(L)$.
\end{remark}

\subsection{Proof of Theorem \ref{FirstMainTheorem}}\label{theproof}

\begin{proof}
If $q = 1$, then $R_{\overline\rho}^\square$ and $R_{\det\overline\rho}^\square$ are formally smooth, so there is nothing to prove. Assume $q > 2$. Then Theorem \ref{normal} says that $\mathcal{X}_{\overline\rho}^\square$ is normal. If $p > n$, the set $\mu_q(K)$ classifies the connected components of both $X_{\det\overline\rho}^\square$ and $X_{\overline\rho}^\square$, and this classification is visibly compatible with respect to the determinant map $d: X_{\det\overline\rho}^\square \to X_{\overline\rho}^\square$. Since the preimage of a connected component is the union of connected components, $d$ must induce a bijection between the connected components, which are irreducible.
\end{proof}

\section{Crystalline Density}\label{crystallinesection}

In this section we prove that the crystalline closed points of $X^\square_{\overline\rho}$ (i.e. closed points whose induced $p$-adic representations are crystalline) are Zariski dense. In particular, we show that each irreducible component contains a crystalline point, and the work of Nakamura in \cite{Nak} (following Chenevier and Gouv\^{e}a-Mazur) shows that the Zariski closure of the crystalline points is the union of some collection of the irreducible components. However, Nakamura assumes that $\mathrm{End}_{k_L[G_K]}(\overline\rho) = k$ and works with the unframed deformation space, so we need to slightly modify the arguments to work in the framed case. To do so, we use work of Breuil-Hellmann-Schraen in \cite{BHS} (and related papers) on the \textit{trianguline deformation space} that is related to the finite slope subspace used by Nakamura, and is already framed.

Note throughout this section we will actually work with the \textit{rigid} generic fiber $\mathcal{X}^{\square}_{\overline\rho}$.

\subsection{Review of Trianguline Deformation Theory}

Here we recall the definition of the trianguline deformation space, following \cite{BHS}. From now on, assume that $L$ contains the Galois closure of $K$, i.e. $L$ contains every embedding of $K$ into $\overline{\mathbf{Q}_p}$. Let $f = [K_0:\mathbf{Q}_p]$, where $K_0$ is the maximal unramified extension of $\mathbf{Q}_p$ contained in $K$. Let $\mathsf{Rig}_L$ denote the category of rigid analytic spaces over $L$. We define $\mathcal{T}: \mathsf{Rig}_L \to \mathsf{Ab}$ taking
	\[ \mathcal{T}(X) = \{\text{continuous characters } K^\times \to \mathscr{O}_X(X)^\times\} \]
whose group structure is given by multiplication of characters. This is represented by a rigid analytic group variety, also denoted $\mathcal{T} \in \mathsf{Rig}_L$. If $\mathbf{h} \in \mathbf{Z}^{\mathrm{Hom}(K,L)}$, then we define the \textit{algebraic} character
	\[ (\cdot)^{\mathbf{h}}: K^\times \mapsto L^\times, z \mapsto \prod_{\tau: K \hookrightarrow L} \tau(z)^{h_\tau}. \]
Fix a uniformizer $\varpi_K$ of $K$. Let $|\cdot|_K$ denote the normalized $\varpi_K$-adic absolute value on $K$, and let $\mathrm{val}_K$ denote the corresponding valuation. Then the set of $L$-points
	\[ \{(\cdot)^{\mathbf{-h}}, |\cdot|_K(\cdot)^{\mathbf{h+1}} : \mathbf{h} \in (\mathbf{Z}_{\geq0})^{\mathrm{Hom}(K,L)}\} \]
is Zariski closed in $\mathcal{T}$, and we define $\mathcal{T}_{\mathrm{reg}}$ to be its open complement inside $\mathcal{T}$. Further, let $\mathcal{T}_{\mathrm{reg}}^n$ denote the Zariski open subset of $\mathcal{T}^n$ defined by
	\[ \mathcal{T}_{\mathrm{reg}}^n = \{(\delta_1,\dots,\delta_n) : \delta_i/\delta_j \in \mathcal{T}_{\mathrm{reg}} \text{ whenever } i \neq j\}. \]

The theory of trianguline representations is built on the theory of $(\varphi,\Gamma_K)$-modules over the Robba ring, where $\Gamma_K = \mathrm{Gal}(K(\mu_p^\infty)/K)$. Denote by
	\[ R_K = \{f(z) \in K\llbracket z,z^{-1} \rrbracket : f \text{ converges on } \{z \in \mathbf{C}_p : r < |z| < 1\} \text{ for some } 0 < r < 1\} \]
the usual Robba ring with coefficients in $K$, and if $A$ is a $K$-module-finite $K$-algebra, then define the relative Robba ring $R_{A,K} := R_K \otimes_{\mathbf{Q}_p} A$. If $L'/L$ is a finite extension and $\rho: G_K \to \mathrm{GL}_n(L')$ is a Galois representation, we denote by $\mathbf{D}_{\mathrm{rig}}(\rho)$ the associated $(\varphi,\Gamma_K)$-module over $R_{L',K}$.

\begin{proposition}[{{\cite[Construction 6.2.4 and Theorem 6.2.14]{KPX}}}]\label{rankonephigamma}
If $L'/L$ is a finite extension, then there is a canonical bijection $\delta \mapsto R_{L',K}(\delta)$ between $\mathcal{T}(L')$ and the set of isomorphism classes of rank 1 $(\varphi,\Gamma_K)$-modules over $R_{L',K}$.

\end{proposition}

\begin{definition}
Let $\rho: G_K \to \mathrm{GL}_n(L')$ be a continuous representation, and fix an $n$-tuple $\delta = (\delta_1,\dots,\delta_n) \in \mathcal{T}^n(L')$. Then we say that $\rho$ is \textit{trianguline of parameter $\delta$} if the $(\varphi,\Gamma_K)$-module $\mathbf{D}_{\mathrm{rig}}(\rho)$ admits a full flag $\mathrm{Fil}^\bullet$ of sub-$(\varphi,\Gamma_K)$-modules (which are free and direct summands as $R_{L',K}$-modules) such that each graded piece $\mathrm{Fil}^i/\mathrm{Fil}^{i-1}$ is isomorphic to $R_{L',K}(\delta_i)$ (c.f. Lemma \ref{rankonephigamma}). We call $\mathrm{Fil}^\bullet$ a \textit{triangulation} of $\rho$.
\end{definition}

Suppose $\rho: G_K \to \mathrm{GL}_n(L')$ is crystalline. Then the associated filtered $\varphi$-module $\mathbf{D}_{\mathrm{cris}}(\rho)$ is finite free rank of rank $n$ over $(K_0 \otimes_{\mathbf{Q}_p} L') \cong \bigoplus_{\tau: K_0 \hookrightarrow L'} L'$, and breaks up under this identification as
	\[ \mathbf{D}_{\mathrm{cris}}(\rho) = \bigoplus_{\tau: K_0 \hookrightarrow L'} \mathbf{D}_{\mathrm{cris},\tau}(\rho) \]
where $\mathbf{D}_{\mathrm{cris},\tau}(\rho) = (B_{\mathrm{cris}} \otimes_{K_0,\tau} \rho)^{G_K}$, such that $\varphi^f$ acts $L'$-linearly on each $\mathbf{D}_{\mathrm{cris},\tau}$. But these $\tau$-labeled $\varphi^f$ operators are all mutually conjugate, so the characteristic polynomial of $\varphi^f$ is independent of $\tau$, and thus we get a multi-set of $\varphi^f$-eigenvalues $\{\alpha_1,\dots,\alpha_n\}$ living in $\overline{\mathbf{Q}_p}$, independent of $\tau: K_0 \hookrightarrow L'$.

There is an alternative characterization of triangulations of a trianguline crystalline representation.

\begin{definition}
Let $\rho: G_K \to \mathrm{GL}_n(L')$ be a crystalline representation. Then a \textit{refinement} of $\rho$ is a full filtration $F^\bullet$ of $\mathbf{D}_{\mathrm{cris}}(\rho)$ by free and $\varphi$-stable $(K_0 \otimes_{\mathbf{Q}_p} L')$-modules.
\end{definition}

In fact, refinements are the same as triangulations: given a triangulation $\mathrm{Fil}^\bullet$ of $\rho$, one can construct a refinement
	\[ F^i := (\mathrm{Fil}^i[\frac1t])^{\Gamma_K}, \]
(where $t = \log(1+z) \in R_K$ is the usual crystalline period: recall $\mathbf{D}_{\mathrm{cris}}(\rho) = (\mathbf{D}_{\mathrm{rig}}(\rho)[\frac1t])^{\Gamma_K}$) and \cite[Proposition 2.4.1]{BeCh}\footnote{In \cite{BeCh} this is proven when $K = \mathbf{Q}_p$, but the proof generalizes straightforwardly to our case: the point is that an ordering on the Hodge-Tate weights and on the $\varphi^f$-eigenvalues determine the parameter $\delta$ of the triangulation.} shows that this map is a bijection.

In particular, suppose the $\varphi^f$-eigenvalues $\{\alpha_i\}$ are distinct. By enlarging $L'$, we may assume that $\{\alpha_i\} \subset L'$. By picking a corresponding eigenvector in $\mathbf{D}_{\mathrm{cris},\tau}(\rho)$ for each $\tau: K_0 \hookrightarrow L'$ and putting them all together, we get a $(K_0 \otimes_{\mathbf{Q}_p} L')$-linear decomposition
	\[ \mathbf{D}_{\mathrm{cris}}(\rho) = (K_0 \otimes_{\mathbf{Q}_p} L')e_1 \oplus \dots \oplus (K_0 \otimes_{\mathbf{Q}_p} L')e_n, \]
such that $\varphi^f(e_i) = \alpha_ie_i$. Since refinements are required to be $\varphi$-stable, every refinement must be of the form
	\[ F^i_\sigma = \bigoplus_{j=1}^i (K_0 \otimes_{\mathbf{Q}_p} L')e_{\sigma(j)} \]
for some $\sigma \in \Sigma_n$, where $\Sigma_n$ is the symmetric group on $n$ letters. We denote the corresponding triangulation $\mathrm{Fil}^\bullet_\sigma$ for each $\sigma \in \Sigma_n$: these are all of the triangulations.

\begin{lemma}[{{\cite[Lemma 2.1]{Smoothness}}}]\label{uniquetriangulation}
If $\rho: G_K \to \mathrm{GL}_n(L')$ is crystalline and trianguline of parameter $\delta$, then there is some ordering $(h_{\tau,1},\dots,h_{\tau,d})_{\tau: K \hookrightarrow L'}$ of the labeled Hodge-Tate weights of $\rho$ and some ordering $(\alpha_1,\dots,\alpha_n)$ of the $\varphi^f$-eigenvalues such that
	\[ \delta_i = (\cdot)^{\mathbf{h}_i}\mathrm{unr}(\alpha_i) \]
where $\mathrm{unr}(\alpha_i)$ is the unramified character of $K^\times$ taking $\varpi_K$ to $\alpha_i$.
\end{lemma}

In particular, if the $\varphi^f$-eigenvalues are all distinct, then there exists a unique triangulation of $\rho$ with parameter $\delta$.

We note a few useful types of crystalline representation. Our definition of Hodge-Tate weights is normalized so that the cyclotomic character has weight $+1$.

\begin{definition}\label{benign}
Let $\rho: G_K \to \mathrm{GL}_n(L')$ be a crystalline representation with $\tau$-labeled Hodge-Tate weights $\{h_{\tau,1} \geq \dots \geq h_{\tau,n}\}_{\tau: K \hookrightarrow L'}$ and $\varphi^f$-eigenvalues $\{\alpha_1,\dots,\alpha_n\}$.
\begin{itemize}
	\item If $h_{\tau,i} \neq h_{\tau,j}$ for all $i \neq j$ and all $\tau$, then we say $\rho$ is \textit{regular} or \textit{regular crystalline}.
	\item If $\alpha_i \neq \alpha_j$ for all $i \neq j$, we say $\rho$ is \textit{$\varphi^f$-generic}.
	\item If $\rho$ is regular and $\mathrm{Fil}^\bullet$ is a triangulation of $\rho$ such that the Hodge-Tate weights of $\mathrm{Fil}^i$ are exactly $\{h_{\tau,1} >  \dots > h_{\tau,i}\}$ for each $\tau$, we say that $\mathrm{Fil}^\bullet$ is \textit{noncritical}.
	\item If $\rho$ is regular and every triangulation of $\rho$ is noncritical, then we say that $\rho$ is \textit{noncritical}.
	\item If $\rho$ is regular, $\varphi^f$-generic, and noncritical, and if additionally $\alpha_i \neq p^{\pm f}\alpha_j$ for all $i \neq j$, then we say that $\rho$ is \textit{benign}.
\end{itemize}
\end{definition}

Now recall that a point of $\mathcal{X}_{\overline\rho}^\square$ is the same as a surjection $f: R_{\overline\rho}^\square[1/p] \twoheadrightarrow L'$ for some finite extension $L'/L$, which gives rise to a $p$-adic representation $\rho_f: G_K \to \mathrm{GL}_n(L')$. We define the subset
	\[ \mathcal{U}_{\mathrm{tri},\overline\rho}^\square = \{(f,\delta) \in \mathcal{X}_{\overline\rho}^\square \times_L \mathcal{T}_{\mathrm{reg}}^n : \rho_f \text{ is trianguline of parameter } \delta\}. \]
Then the \textit{trianguline deformation space $\mathcal{X}_{\mathrm{tri},\overline\rho}^\square$} is the Zariski closure of $\mathcal{U}_{\mathrm{tri},\overline\rho}^\square$ in $\mathcal{X}_{\overline\rho}^\square \times_L \mathcal{T}_{\mathrm{reg}}^n$.

Recall that at any point $f \in \mathcal{X}_{\overline\rho}^\square$, the completion $\mathscr{O}_{\mathcal{X}_{\overline\rho}^\square,f}^\wedge$ is the universal framed deformation ring for the induced representation $\rho_f: G_K \to \mathrm{GL}_n(L')$. In fact, a similar result holds for a point $x = (f_x,\delta_x) \in \mathcal{X}_{\mathrm{tri},\overline\rho}^\square$ such that $\rho_x := \rho_{f_x}$ is benign, as we recall now.

Since $\rho_x$ is in particular $\varphi^f$-generic, Lemma \ref{uniquetriangulation} says that there is a \textit{unique} triangulation $\mathrm{Fil}_x^\bullet$ of $\rho_x$ with parameter $\delta_x$. So we define the deformation problem
	\[ D^\square_{\rho_x,\mathrm{Fil}^\bullet_x[\frac1t]}: \mathsf{Art}_{L'} \to \mathsf{Set} \]
taking an Artinian local $L'$-algebra $A$ to the set of pairs $(\rho, \mathrm{Fil}^\bullet)$ where $\rho: G_K \to \mathrm{GL}_n(A)$ of $\rho$ lifts $\rho_x$, and $\mathrm{Fil}^\bullet$ is a full filtration (by direct summands) of $\mathbf{D}_{\mathrm{rig}}(\rho)[\frac1t]$ by finite free $R_{A,K}$-modules whose successive quotients are of the form $R_{A,K}(\delta)[\frac1t]$, and which lifts the filtration $\mathrm{Fil}^\bullet_x[\frac1t]$.

\begin{proposition}
The natural forgetful map $D^\square_{\rho_x,\mathrm{Fil}^\bullet_x[\frac1t]} \to D^\square_{\rho_x}$ is injective and relatively representable. We denote the representing ring $R^\square_{\rho_x,\mathrm{Fil}^\bullet_x[\frac1t]}$.
\end{proposition}
\begin{proof}
Since $\rho_x$ is $\varphi^f$-generic, $\delta_x$ is regular in the sense of \cite[(3.1)]{HMS}. The result then follows from \cite[Proposition 3.5]{HMS} and the identification (see \cite[Section 3.6]{LocalModel} for more details)
	\[ D^\square_{\rho_x,\mathrm{Fil}^\bullet_x[\frac1t]} \cong D^\square_{\rho_x} \times_{D_{\mathbf{D}_{\mathrm{rig}}(\rho_x)}} D_{\mathbf{D}_{\mathrm{rig}}(\rho_x),\mathrm{Fil}^\bullet_x[\frac1t]}, \]
\end{proof}

\begin{proposition}\label{triangulinedeformations}
There is an isomorphism
	\[ \mathscr{O}_{\mathcal{X}^\square_{\mathrm{tri},\overline\rho},x}^\wedge \cong R^\square_{\rho_x,\mathrm{Fil}^\bullet_x[\frac1t]} \]
of integral domains.
\end{proposition}
\begin{proof}
Let $G = \Spec L' \times_{\Spec \mathbf{Q}_p} \mathrm{Res}_{K/\mathbf{Q}_p} \mathrm{GL}_{n,K} \cong \prod_{\tau: K \hookrightarrow L'} \mathrm{GL}_{n,L'}$ with Weyl group $W = \prod_{\tau: K \hookrightarrow L'} \Sigma_n$. In \cite{LocalModel}, the authors show that one can associate to $(\rho_x,\mathrm{Fil}^\bullet_x[\frac1t])$ an element $w_x \in W$ such that the irreducible components of $\Spec R^\square_{\rho_x,\mathrm{Fil}^\bullet_x[\frac1t]}$ are in bijection with the set $\{w \in W : w \succeq w_x\}$, where $\succeq$ denotes the Bruhat ordering on $W$ (see Theorem 3.6.2 and the proof of Corollary 4.3.2). Roughly speaking, $w_x$ measures the relative position of $\mathrm{Fil}^\bullet_x[\frac1t]$ with respect to the Hodge filtration on $\mathbf{D}_{\mathrm{dR}}(\rho_x)$.

However, the assumption that $\rho_x$ is noncritical ensures that $w_x$ is the maximal element for the Bruhat ordering, i.e. each $w_{x,\tau}$ is the order-reversing permutation. Therefore, $\Spec R^\square_{\rho_x,\mathrm{Fil}^\bullet_x[\frac1t]}$ is irreducible, so \cite[Corollary 3.7.8]{LocalModel} exactly says that
	\[ \mathscr{O}_{\mathcal{X}^\square_{\mathrm{tri},\overline\rho},x}^\wedge \cong R^\square_{\rho_x,\mathrm{Fil}^\bullet_x[\frac1t]}, \]
and \cite[Theorem 3.6.2(a)]{LocalModel} says that this ring is reduced.
\end{proof}

\subsection{Proof of Density}

Now we prove the main theorem.

\begin{definition}
Let
	\[ \mathcal{X}^\square_{\mathrm{reg}-\mathrm{cris}} = \{x \in \mathcal{X}^\square_{\overline\rho} : \rho_x \text{ is regular crystalline}\} \]
and let $\overline{\mathcal{X}}^\square_{\mathrm{reg}-\mathrm{cris}}$ be its Zariski closure in $\mathcal{X}_{\overline\rho}^\square$. Let
	\[ \mathcal{X}_{\mathrm{cr}} = \{(f',\delta') \in \mathcal{X}_{\mathrm{tri},\overline\rho}^\square : \rho_{f'} \text{ is regular crystalline, $\varphi^f$-generic, and noncritical} \} \]
\end{definition}

We will need the following result about density of crystalline points in the trianguline deformation space:

\begin{proposition}[{{\cite[Proposition 4.1.4]{LocalModel}}}]\label{densityeigenvariety}
Suppose $x = (f_x,\delta_x) \in \mathcal{X}_{\mathrm{tri},\overline\rho}^\square$ is benign. Then there exists an affinoid open neighborhood $\mathcal{U} \subseteq \mathcal{X}^\square_{\mathrm{tri},\overline\rho}$ of $x$ such that $\mathcal{U} \cap \mathcal{X}_{\mathrm{cr}}$ is Zariski dense in $\mathcal{U}$.
\end{proposition}

\begin{proposition}\label{RegularCrystallineDense}
$\overline{\mathcal{X}}^\square_{\mathrm{reg}-\mathrm{cris}}$ is a union of irreducible components of $\mathcal{X}^\square_{\overline\rho}$.
\end{proposition}
\begin{proof}
Let $\mathcal{Z}$ be an irreducible component of $\overline{\mathcal{X}}^\square_{\mathrm{reg}-\mathrm{cris}}$. Note the singular locus $\mathcal{Z}_{\mathrm{sing}} \subset \mathcal{Z}$ is a proper Zariski closed subset, so its complement $\mathcal{U}$ is an admissible open in $\mathcal{Z}$. Thus, we may pick a smooth point $x \in \mathcal{U}$ with corresponding representation $\rho_x: G_K \to \mathrm{GL}_n(L')$, and by density we assume $x \in \mathcal{X}^\square_{\mathrm{reg}-\mathrm{cris}}$. By Corollary 2.7.7\footnote{In fact, we use a slight modification: we first use the space of semi-stable deformations with fixed Hodge-Tate weights constructed in \cite[Corollary 2.6.2]{Kis08}, and then consider the zero locus of the monodromy operator $N$.} in \cite{Kis08}, there is a Zariski closed subspace $\mathcal{X}_{\overline\rho,\mathrm{cris}}^{\square,\mathbf{k}_x} \subset \mathcal{X}_{\overline\rho}^\square$ consisting of the crystalline representations with Hodge-Tate weights $\mathbf{k}_x$, where $\mathbf{k}_x$ denotes the Hodge-Tate weights of $\rho_x$, as above. By \cite[Lemma 4.2]{Nak} (wit 	h $U = \mathcal{U} \cap \mathcal{X}_{\overline\rho,\mathrm{cris}}^{\square,\mathbf{k}_x}$) we may assume $\rho_x$ is benign.

In fact, $x$ is smooth in $\mathcal{X}_{\overline\rho}^\square$. To see this, note that
	\[ \mathscr{O}_{\mathcal{X}_{\overline\rho}^\square,x}^\wedge \cong R_{\rho_x}^\square, \]
where $R_{\rho_x}^\square$ is the universal deformation ring of the framed deformation problem for $\rho_x$. Thus, it suffices to show that $H^2(G_K,\ad\rho_x) = 0$, which is the same as showing that $\mathrm{Hom}_{\kappa(\rho_x)[G_K]}(\rho_x,\rho_x\otimes\epsilon) = 0$ by local Tate duality. But a morphism $g: \rho_x \to \rho_x \otimes \epsilon$ induces a $\varphi^f$-equivariant map $\mathbf{D}_{\mathrm{cris}}(\rho_x) \to \mathbf{D}_{\mathrm{cris}}(\rho_x \otimes \epsilon)$. If $\{\alpha_i\}_i$ denotes the set of (distinct) eigenvalues of $\varphi^f$ on $\mathbf{D}_{\mathrm{cris}}(\rho_x)$, then the eigenvalues of $\varphi^f$ on $\mathbf{D}_{\mathrm{cris}}(\rho_x \otimes \epsilon)$ are exactly $\{p^f\alpha_i\}_i$. But $\rho_x$ is benign, so in particular, $\alpha_i \neq p^{\pm f}\alpha_j$ for $i \neq j$, and thus $g = 0$.

The irreducible set $\mathcal{Z}$ admits a closed immersion $\mathcal{Z} \hookrightarrow \mathcal{V}$ where $\mathcal{V}$ is an irreducible component of $\mathcal{X}_{\overline\rho}^\square$. We wish to show that this map is an equality, for which it suffices to show that $\dim \mathcal{Z} = \dim \mathcal{V}$. Since dimension can be computed as the dimension of the tangent space at smooth points, it suffices to show that the natural injection
	\[ T_x \mathcal{Z} \hookrightarrow T_x \mathcal{X}^\square_{\overline\rho} = T_x \mathcal{V} \]
is an isomorphism.

As noted above, the triangulations of $\rho_x$ are exactly parametrized by $\sigma \in \Sigma_n$, and we write them as $\mathrm{Fil}^\bullet_\sigma$. Each pair $(x,\mathrm{Fil}^\bullet_\sigma)$ defines a point $y_\sigma \in \mathcal{X}_{\mathrm{tri},\overline\rho}^\square$. Since $\rho_x$ is benign, Proposition \ref{triangulinedeformations} says that $\mathscr{O}_{\mathcal{X}^\square_{\mathrm{tri},\overline\rho},y_\sigma}^\wedge \cong R^\square_{\rho_x,\mathrm{Fil}^\bullet_\sigma[\frac1t]}$ is an integral domain. This implies that there is a unique irreducible component of $\mathcal{X}^\square_{\mathrm{tri},\overline\rho}$ containing $y_\sigma$, which we call $\mathcal{Y}_\sigma$. By Proposition \ref{densityeigenvariety}, there exists an affinoid open neighborhood $\mathcal{U}_\sigma$ of $y_\sigma$ such that $\mathcal{U}_\sigma \cap \mathcal{X}_{\mathrm{cr}}$ is dense in $\mathcal{U}_\sigma$. But then $\mathcal{U}_\sigma \cap \mathcal{X}_{\mathrm{cr}} \cap \mathcal{Y}_\sigma$ is dense in $\mathcal{U}_\sigma \cap \mathcal{Y}_\sigma$, which is a nonempty open and thus dense in $\mathcal{Y}_\sigma$. Thus $\mathcal{X}_{\mathrm{cr}} \cap \mathcal{Y}_\sigma$ is dense in $\mathcal{Y}_\sigma$.

Under the natural projection
	\[ \mathcal{X}_{\mathrm{tri},\overline\rho}^\square \to \mathcal{X}_{\overline\rho}^\square \times_L \mathcal{T}_{\mathrm{reg}}^n \to \mathcal{X}_{\overline\rho}^\square, \]
the subset $\mathcal{X}_{\mathrm{cr}}$ lands in $\mathcal{X}_{\mathrm{reg}-\mathrm{cris}}$ and $\mathcal{Y}_\sigma$ lands in $\mathcal{V}$. By density, this descends to a map $\mathcal{Y}_\sigma \to \mathcal{Z}$, taking $y_\sigma \mapsto x$. Therefore, by considering all $\sigma \in \Sigma_n$, we get maps
	\[ \bigoplus_{\sigma \in \Sigma_n} T(R^\square_{\rho_x,\mathrm{Fil}^\bullet_\sigma[\frac1t]}) = \bigoplus_{\sigma \in \Sigma_n} T_{y_\sigma} \mathcal{Y}_\sigma \to T_x \mathcal{Z} \hookrightarrow T_x \mathcal{V} \cong T(R_{\rho}^\square). \]
which are actually induced by the forgetful maps $D_{\rho_x,\mathrm{Fil}^\bullet_\sigma[\frac1t]}^\square \to D_{\rho_x}^\square$.

Let $D_{\rho_x}$ and $D_{\rho_x,\mathrm{Fil}^\bullet_\sigma[\frac1t]}$ be the unframed deformation functors of $\rho_x$ and $(\rho_x,\mathrm{Fil}^\bullet_\sigma[\frac1t])$ respectively. Let $\pi_\sigma: D_{\rho_x,\mathrm{Fil}^\bullet_\sigma[\frac1t]} \to D_{\rho_x}$ denote the map forgetting the filtration. Then $D^\square_{\rho_x} \to D_{\rho_x}$ is formally smooth of relative dimension $e = n^2 - \dim H^0(G_K,\ad\rho_x)$. So we may write $T(R^\square_{\rho_x}) = T(D_{\rho_x}) \oplus V^\square$ where $V^\square$ is some $e$-dimensional $L'$-vector space, and such that
	\[ T(R^\square_{\rho_x}) = T(D_{\rho_x}) \oplus V^\square \twoheadrightarrow T(D_{\rho_x}) \]
is just the natural projection map. There is a canonical isomorphism of functors
	\[ D^\square_{\rho_x,\mathrm{Fil}^\bullet_\sigma[\frac1t]} \cong D^\square_{\rho_x} \times_{D_{\rho_x}} D_{\rho_x,\mathrm{Fil}^\bullet_\sigma[\frac1t]}, \]
so the map $\bigoplus_{\sigma \in \Sigma_n} T(R^\square_{\rho_x,\mathrm{Fil}^\bullet_\sigma[\frac1t]}) \to T(R^\square_{\rho_x})$ factors as
	\[ \bigoplus_{\sigma \in \Sigma_n} T(R^\square_{\rho_x,\mathrm{Fil}^\bullet_\sigma[\frac1t]}) \cong \bigoplus_{\sigma \in \Sigma_n} (T(D_{\rho_x,\mathrm{Fil}^\bullet_\sigma[\frac1t]}) \oplus V^\square) \xrightarrow{\sum (T(\pi_\sigma) \oplus \mathrm{id}_{V^\square})} T(D_{\rho_x}) \oplus V^\square \cong T(R^\square_{\rho_x}) \]
By \cite[Corollary 3.13]{HMS}, $\sum_{\sigma \in \Sigma_n} T(\pi_\sigma)$ is surjective, and the above map $\bigoplus_{\sigma \in \Sigma_n} V^\square \xrightarrow{\sum \mathrm{id}_{V^\square}} V^\square$ is clearly surjective, so we conclude that the map is surjective, which is exactly what we wanted. Thus, $T_x \mathcal{Z} \hookrightarrow T_x \mathcal{V}$ is an isomorphism.
\end{proof}

\begin{theorem}\label{SecondMainTheorem}
Under the conditions described in the statement of Theorem \ref{secondref},
	\[ \overline{\mathcal{X}}^\square_{\mathrm{reg}-\mathrm{cris}} = \mathcal{X}^\square_{\overline\rho} \]
\end{theorem}
\begin{proof}
By Proposition \ref{RegularCrystallineDense}, it suffices to show the existence of a regular crystalline point in each irreducible component of $\mathcal{X}^\square_{\overline\rho}$: by restricting algebraic characters to $\mathscr{O}_K^\times$, we just construct them explicitly.

Let $\mathrm{rec}_K: \widehat{K^\times} \xrightarrow\sim G_K^{\mathrm{ab}}$ be the local reciprocity map, normalized so that $\varpi_K$ is sent to a lift of the geometric Frobenius. Let $I = \{i \in \mathbb{Z} : 0 < i < q \text{ and } \gcd(i,q) = 1\}$ and $J = \{1,\dots,|\mathrm{Hom}_{\mathbf{Q}_p(\mu_q)}(K,L)|\}$, and pick some bijection $\iota: \mathrm{Hom}(K,L) \xrightarrow\sim I \times J$ such that if $\tau_1|_{\mathbf{Q}_p(\mu_q)} = \tau_2|_{\mathbf{Q}_p(\mu_q)}$, then $p_I(\iota(\tau_1)) = p_I(\iota(\tau_2))$, where $p_I: I \times J \to I$ is the projection. Then for $\mathbf{h} = (h_{i,j}) \in \mathbf{Z}_{> 0}^{I \times J}$, we define the character $\chi^\mathbf{h}: G_K^{\mathrm{ab}} \to L^\times$ by setting $\chi^\mathbf{h}(\mathrm{rec}_K(\varpi_K)) := 1$, and setting $(\chi^\mathbf{h} \circ \mathrm{rec}_K)|_{\mathscr{O}_K^\times} := (\cdot)^{\iota^*(\mathbf{h})}|_{\mathscr{O}_K^\times}$ where $(\cdot)^{\iota^*(\mathbf{h})}$ is the algebraic character in $\mathcal{T}(L)$ defined at the beginning of this subsection. This is crystalline with labeled Hodge-Tate weights $\iota^*(\mathbf{h})$. 

The point is that if $\zeta_0 \in K$ is some choice of primitive $q$th root of unity and $\tau_0 \in \mathrm{Hom}(K,L)$ is some embedding such that $p_I(\iota(\tau_0)) = 1$, then
	\[ \chi^\mathbf{h}(\mathrm{rec}_K(\zeta_0)) = \tau_0(\zeta_0)^{\sum_{(i,j) \in I \times J} ih_{i,j}}. \]
So pick some $\mathbf{h} \in \mathbf{Z}_{>0}^{I \times J}$ such that $\sum_{(i,j) \in I \times J} ih_{i,j}$ is coprime to $q$: this is possible since $(i,q) = 1$ for all $i \in I$. Then $\chi^\mathbf{h}(\mathrm{rec}_K(\zeta_0))$ is actually a primitive $q$th root of unity. Furthermore,
	\[ (\chi^\mathbf{h})^{\otimes (p^f-1)} \equiv 1 \mod \varpi_L, \]
which implies that $(\chi^\mathbf{h})^{\otimes (p^f-1)}$ is induced by a point in the generic fiber $\mathcal{X}_\mathbf{1}$ of the deformation space of the trivial character.

Let $\chi_0 = (\chi^\mathbf{h})^{\otimes(p^f-1)}$. For $m = 0,\dots,q-1$ define
	\[ \rho_m = \left(\bigoplus_{i=1}^{n-1} \chi_0^{\otimes i}\right) \oplus \chi_0^{\otimes(n+m)}. \]
Then $\rho_m$ is induced by a point in $\mathcal{X}^\square_{\overline\rho}$. It is crystalline and its $\tau$-labeled Hodge-Tate weights are exactly
	\[ (p^f-1)h_\tau,2(p^f-1)h_\tau,\dots,(n-1)(p^f-1)h_\tau,(n+m)(p^f-1)h_\tau, \]
which are distinct and nonzero since $h_\tau > 0$ by assumption (here $h_\tau := \iota^*(\mathbf{h})_\tau$) so $\rho_m$ is actually regular. Finally note that
	\[ \{\det(\rho_m(\mathrm{rec}_K(\zeta_0)))\}_{m=0, \dots, q-1} = \mu_q(L). \]
By Theorem \ref{FirstMainTheorem} and Remark \ref{lubintate}, we have thus found regular crystalline points in each of the irreducible components of $\mathcal{X}^\square_{\overline\rho}$.
\end{proof}

\bibliography{refs}

\end{document}